\documentclass[12pt]{amsart}

\usepackage{upgreek }
\usepackage{graphicx}
\usepackage{amscd}
\usepackage{amsmath}
\usepackage{amsfonts}
\usepackage{amssymb}
\usepackage{mathrsfs}
\usepackage{comment}
\excludecomment{mycomment}
\usepackage[usenames, dvipsnames]{color}

\usepackage{pdfsync}
\usepackage{bbm}
\usepackage{dsfont}
\usepackage{hyperref}
 \hypersetup{
     colorlinks=true,
     linkcolor=blue,
     filecolor=blue,
     citecolor = red,
     urlcolor=cyan,
     }
\usepackage{enumerate}
\usepackage{booktabs}
\usepackage{float}
\usepackage{wrapfig}
\usepackage{caption}
\usepackage{subcaption}
\usepackage{longtable}
\usepackage{listings}
\usepackage[T1]{fontenc}
\usepackage[scaled]{beramono}
\usepackage{tikz}
\usepackage{pgffor}
\usepackage[all]{xy}
\usepackage{bbold}
\usepackage{enumerate}
\definecolor{trp}{rgb}{1,1,1}

\definecolor{red}{rgb}{1,0,.2}

\definecolor{myred1}{RGB}{255, 0, 0}
\definecolor{myyellow1}{RGB}{255, 255, 219}
\definecolor{mygreen1}{RGB}{0, 255, 0}
\definecolor{mygreen2}{RGB}{0, 126, 0}
\definecolor{myblue1}{RGB}{0, 0, 255}

\usepackage{amsthm}

\theoremstyle{plain}
\newtheorem{theorem}{Theorem}[section]

\newtheorem{claim}[theorem]{Claim}
\newtheorem{fact}[theorem]{Fact}

\newtheorem{corollary}[theorem]{Corollary}

\newtheorem{definition}[theorem]{Definition}
\newtheorem{example}[theorem]{Example}

\newtheorem{lemma}[theorem]{Lemma}

\newtheorem{PA}[theorem]{Principal Assumption}

\newtheorem{proposition}[theorem]{Proposition}

\numberwithin{equation}{section}

\newcommand*{\ind}{\mathbbm{1}}
\newcommand*{\arabicdec}[1]{\the\numexpr\value{#1}\relax}

\usepackage{anysize}
\papersize{24.5cm}{16.2006cm}

\marginsize{1.5cm}{1cm}{0cm}{0cm}

\usepackage{caption}

\definecolor{blue}{rgb}{0,0,1}

\definecolor{Buff}{rgb}{0.94, 0.86, 0.51}

\definecolor{red}{rgb}{1,0,.7}

\usepackage{tikz}
\usetikzlibrary{cd,decorations.pathreplacing,positioning,arrows}

\newcommand{\sci}{\subset}
\newcommand{\set}[1]{\{#1\}}

\newcommand{\ga}{\alpha}
\newcommand{\gb}{\beta}

\newcommand{\gk}{\kappa}

\newcommand{\tit}{\textit}

\newcommand{\C}[1]{\mathcal{#1}}
\newcommand{\D}[1]{\mathbb{#1}}

\newcommand{\te}{\text}

\newcommand{\ol}{\overline}
\newcommand{\ul}{\underline}

\begin{document}
\title[]{Quantization dimension for self-similar measures of overlapping construction}

\author{MRINAL KANTI ROYCHOWDHURY}
\address{ Mrinal Kanti Roychowdhury, University of Texas Rio Grande Valley,
1201 West University Drive, Edinburg, TX 78539-2999, USA} \email{mrinal.roychowdhury@utrgv.edu}

\author{K\'aroly Simon}
\address{K\'aroly Simon, Budapest University of Technology and Economics, MTA-BME Stochastics Research Group, P.O. Box 91, 1521 Budapest, Hungary} \email{simonk@math.bme.hu}

\maketitle
\pagestyle{myheadings}\markboth{Mrinal Kanti Roychowdhury and Karoly Simon}{Quantization dimension for self-similar measures of overlapping construction}

\begin{abstract}
  Quantization dimension has been computed for many invariant measures of dynamically defined fractals having well separated cylinders, that is, in the cases when the so-called Open Set Condition (OSC) holds. To attack the same problem in case of heavy overlaps between the cylinders, we consider a
  family of self-similar measures, for which the underlying Iterated Function System satisfies
   the so-called Weak Separation Property (WSP) but does not satisfy the OSC since complete overlaps occur in between the cylinders.
   The work in this paper also shows that the quantization dimension determined for the set of overlap self-similar construction satisfying the WSP has a relationship with the temperature function of the thermodynamic formalism.
\end{abstract}

\section{Introduction}

The basic goal of quantization for probability distribution is to reduce the number of values, which is typically uncountable, describing a probability distribution to some finite set and thus approximation of a continuous probability distribution by a discrete distribution. It has broad applications in signal processing, telecommunications, data compression, image processing and cluster analysis.  Over the years, many useful theorems have been proved and numerous other results and algorithms have been obtained in quantization.  For a detailed survey on the subject and comprehensive lists of references to the literature one is referred to \cite{B, BW, G, GG, GKL, GL1, GL2, GL4, GN, Z1, Z2}. Rigorous mathematical treatment of the fundamentals of the quantization theory is provided in Graf-Luschgy's book (see \cite{GL1}). In general, these theorems almost exclusively involve absolutely continuous probability measures on $\D R^d$. Two main goals of the theory are: $(1)$ finding the exact configuration of a so-called `$n$-optimal set' which corresponds to the support of the quantized version of the distribution, and $(2)$ estimating the rate at which some specified measure of the error goes to zero as $n$ goes to infinity. This paper deals with the second kind of problem.

Given a Borel
probability measure $\mu$ on $\D R^d$, a number $r \in (0, +\infty)$
and a natural number $n \in \D N$, the $n$th \tit{quantization
error} of order $r$ for $\mu$, is defined by
\begin{equation*}\label{eq0}  V_{n, r}:=V_{n,r}(\mu)=\te{inf}\left\{\int d(x, \ga)^r d
\mu(x): \ga \sci \D R^d, \, 1\leq \te{card}(\ga) \leq n\right\},\end{equation*}
where $d(x, \ga)$ denotes the distance from the point $x$ to the set
$\ga$ with respect to a given norm $\|\cdot\|$ on $\D R^d$. If $\ga$ is a finite set, the error $\int d(x, \ga)^r d
\mu(x) $ is often referred to as the \tit{cost,} or \tit{distortion error} of order $r$ for $\mu$ and $\ga$. It is known that for a Borel probability measure $\mu$,  if its support contains infinitely many elements and $\int \| x\|^r d\mu(x)$ is finite, then an optimal set of $n$-means always has exactly $n$-elements (see \cite{AW, GKL, GL, GL1}). This set $\ga$ can then be used to give a best approximation of
$\mu$ by a discrete probability supported on a set with no more than
$n$ points.  Such a set $\ga$ for which the infimum occurs and contains no more than $n$ points is called an \tit{optimal set of $n$-means}, or \tit{optimal set of $n$-quantizers} (of order $r$). Under suitable conditions this can be done by giving
each point $a \in \ga $ a mass corresponding to $\mu(A_a)$, where
$A_a$ is the set of points $x \in \D R^d$ such that $d(x, \ga)=d(x,
a)$. So, $\set{A_a : a \in \ga}$ is the `Voronoi' partition of $\D R^d$ induced by $\ga$. Of course, the idea of `best approximation' is, in general,
dependent on the choice of $r$.
For some recent work in the direction of optimal sets of $n$-means, one can see \cite{CR, DR1, DR2, GL5, R, R1, R2, R3, R4, R5, R6, RR1}. The set of all optimal sets of $n$-means for a Borel probability measure $\mu$ is denoted by $\C C_{n,r}(\mu)$, i.e.,
\begin{equation} \label{a85}
\C C_{n,r}(\mu):=\set{\ga \sci \D R^d : 1\leq \te{card}(\ga) \leq n \te{ and } V_{n, r}(\mu)=\int d(x, \ga)^r d
\mu(x)}.
\end{equation}
Write $e_{n, r}(\mu):=V_{n, r}^{\frac 1r}(\mu)$.
The numbers
\begin{equation} \label{eq55} \ul D_r(\mu):=\liminf_{n\to \infty}  \frac{\log n}{-\log e_{n,r}(\mu)}, \te{ and } \ol D_r(\mu):=\limsup_{n\to \infty} \frac{\log n}{-\log e_{n, r}(\mu)}, \end{equation}
are called the \tit{lower} and the \tit{upper quantization dimensions} of $\mu$ of order $r$, respectively. If $\ul D_r (\mu)=\ol D_r (\mu)$, the common value is called the \tit{quantization dimension} of $\mu$ of order $r$ and is denoted by $D_r(\mu)$. Quantization dimension measures the speed at which the specified measure of the error goes to zero as $n$ tends to infinity.
For any $\gk>0$, the numbers $\liminf_n n^{\frac r \gk}  V_{n, r}(\mu)$ and $\limsup_n  n^{\frac r \gk}V_{n, r}(\mu)$ are called the \tit{$\gk$-dimensional lower} and \tit{upper quantization coefficients} for $\mu$, respectively. The quantization coefficients provide us with more accurate information about the asymptotics of the quantization error than the quantization dimension. Compared to the calculation of quantization dimension, it is usually much more difficult to determine whether the lower and the upper quantization coefficients are finite and positive. It follows from \cite[Proposition~11.3]{GL1} that if
\begin{equation}\label{w86}
0\leq t<\ul D_r<s, \te{ then }
\lim_{n\to \infty} n e^t_{n,r}=+\infty \te{ and } \liminf_{n\to \infty} n e^s_{n, r}=0,
\end{equation}
 and if
\begin{equation}\label{y97}
0\leq t<\ol D_r<s, \te{ then }
\limsup_{n\to \infty} n e^t_{n,r}=+\infty \te{ and } \lim_{n\to \infty} n e^s_{n, r}=0.
\end{equation}

For probabilities with non-vanishing absolutely continuous parts the numbers $D_r$ are all equal to the dimension $d$ of the underlying space, but for singular probabilities the family $\{D_r\}_{r>0}$ gives an interesting description of their geometric (multifractal) structures. A detailed account of this theory can be found in \cite{GL1}. There the quantization dimension is introduced as a new type of fractal dimension spectrum and a formula for its determination is derived in the case of self-similar probabilities with the strong separation property. In \cite{GL2}, Graf and Luschgy extended the above result and determined the quantization dimension $D_r$ of self-similar probabilities with the weaker open set condition, but there it remained open whether the $D_r$-dimensional lower quantization coefficient is positive. Later they answered it in \cite{GL3}. Under the open set condition, Lindsay and
Mauldin (see \cite{LM}) determined the quantization dimension function $D_r$, where $r\in (0, +\infty)$, of an $F$-conformal measure $m$ associated
with a conformal iterated function system determined by finitely
many conformal mappings. Subsequently, quantization
dimension has been computed for many invariant measures of
dynamically defined fractals having well separated cylinders (that is the
cases when the so-called Open Set Condition holds), for example, one can see \cite{MR, R7, R8, R9, R10, R11, R12, R13, R14, R15}.
From all the known results it can be seen that if the quantization dimension function $D_r(\mu)$ for $r>0$ of a fractal probability measure $\mu$ exists, it has a relationship with the temperature function of the thermodynamic formalism that arises in multifractal analysis of $\mu$ (see Figure~\ref{Fig0}).

In a very recent paper,
among many other interesting applications, S. Zhu \cite{Zhu} has solved the problem of the computation of quantization dimension for the complete overlapping case in the following very special situation: Let $\mathcal{S}=\{S_k\}_{k\leq m}$
be a self-similar IFS with the following properties: $(1)$ there are distinct $i$ and $j$ with $S_i=S_j$. $(2)$ The self-similar IFS $\mathcal{S}=\{S_k\}_{k\leq m, k\ne j}$ satisfies the so-called strong separation property. That is, for any two distinct $u,v\in\{1,\dots,m\}\setminus\{j\}$, $S_u(\Lambda)\cap S_v(\Lambda)=\emptyset$, where $\Lambda$ is the attractor of the IFSs $\mathcal{S}$.
In general, if we are given a self-similar IFS which satisfies the so-called Weak Separation Property (WSP), but does not satisfy the OSC we cannot get rid of the overlapping feature of the system just by throwing out some of the mappings of the IFS. Even if we get rid of one of the mappings which causes the total overlapping feature, we still have overlaps in the remaining system (which are not complete overlaps).

In this paper, we make a step towards our goal to determine the quantization dimension for self-similar measures on the line in the case when the underlying self-similar system satisfies the WSP (for the definition and basic properties see \cite{Zer}).
Namely, we solve this problem for a special family which has the above mentioned properties.  Our work also shows that the quantization dimension determined for a set of overlap self-similar construction satisfying the WSP has a relationship with the temperature function of the thermodynamic formalism.

Recently, Kesseb\"{o}hmer et al. \cite[Corollary 1.12]{KNZ} proved  that the quantization dimension $D_r$ exists for every self-conformal measure, and it is determined by the intersection point of the $L^q$-spectrum $\beta (q)$ of the measure and the line through the origin with slope $r$ as indicated by Figure \ref{Fig0}. In our paper, using a completely different technique for an IFS we also calculated the quantization dimension. Since the $L^q$-spectrum
for the self-similar measure studied in our paper has not been described explicitly before, our result is different and has its importance because of the different techniques of the work. Moreover, the combination of \cite[Corollary 1.2]{KNZ} and our main result Theorem \ref{x70} yields an explicit formula
(see Corollary \ref{r97})
for the $L^q$-spectrum
$\beta (q)$, $q\in(0,1)$ for the self-similar measure with overlaps studied in this paper.

\begin{figure}
\vspace{-1.4 in}
\centerline{\includegraphics[width=8 in, height=4.2 in]{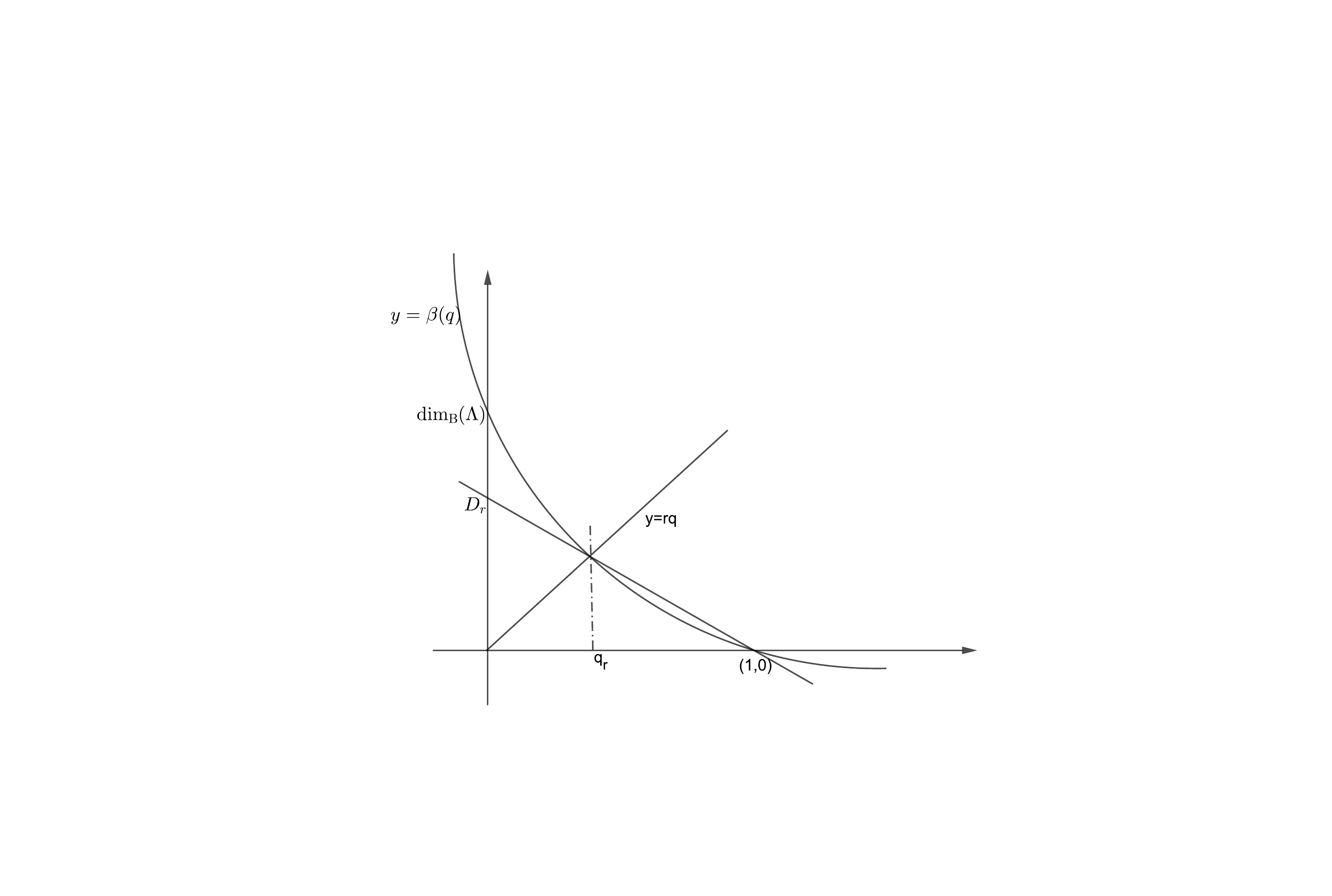}}
\vspace{-1. in}
\caption {\footnotesize{To determine $D_r$ first find the point of intersection of $y=\gb(q)$ and the line $y=rq$. Then, $D_r$ is the $y$-intercept of the line through this point and the point $(1, 0)$.}}  \label{Fig0}
\vspace{-0.4 in}
\end{figure}

\section{An overlapping self-similar IFS on the line}

We consider the following self-similar IFS on $\mathbb{R}$
\begin{equation}\label{xx81}
  \mathcal{S}=\left\{S_i(x)=\frac{1}{3}x+i\right\}_{i\in\left\{0,1,3\right\}}.
\end{equation}

Naturally the alphabet
corresponds to this IFS is
$\mathcal{A}:=\left\{ 0,1,3 \right\}$.
We write $\Sigma $  ($\Sigma ^*$) for the set of infinite (finite) words, respectively, over the alphabet $\mathcal{A}$.
As usual we write $\sigma $ for the left shift on
$\Sigma \cup\Sigma ^*$.
We write
\begin{equation}
\label{y73}
\mathbf{i}^-:=(i_1,\dots  ,i_{n-1}) \quad
\text{ for an }\quad
\mathbf{i}=(i_1,\dots  ,i_n)\in\mathcal{A}^n.
\end{equation}
We say that
$\Gamma \subset \Sigma ^*$ is a maximal finite antichain if for every $\mathbf{i}\in\Sigma $ there exists a unique $n$ such that $\mathbf{i}|_n\in\Gamma $.

Let $\Lambda$ be the attractor of $\mathcal{S}$. That is,
$\Lambda $ is the unique non-empty compact set satisfying $\Lambda =\bigcup\limits_{i\in\mathcal{A}} S_i(\Lambda )$.
The smallest interval that contains $\Lambda $ is $I=\left[ 0,\frac{9}{2} \right]$.
Put $I_{i_1\dots  i_n}:=S_{i_1\dots  i_n}(I)$, where we use the shorthand notation $S_{i_1\dots  i_n}=S_{i_1}\circ\cdots\circ S_{i_n}$.

The natural projection $\Pi :\Sigma \cup\Sigma ^*\to\Lambda  $  is defined by
\begin{equation}
\label{y96}
\Pi (\mathbf{i}):=\sum _{ k=1}^{|\mathbf{i}| }
i_k3^{-(k-1)},
\end{equation}
where $|\mathbf{i}|=n$ if $\mathbf{i}\in \mathcal{A}^n$ and $|\mathbf{i}|=\infty  $ if $\mathbf{i}\in\Sigma $.
For a finite word $\mathbf{i}\in\Sigma ^*$ the projection $\Pi (\mathbf{i})$ is
the left end point of the interval $I_{\mathbf{i}}$.
That is,
\begin{equation}
\label{y91}
\text{ for }\mathbf{i},\mathbf{j}\in \Sigma^* \text{ with }
|\mathbf{i}|=|\mathbf{j}|
\text{ we have }
\Pi (\mathbf{i})=\Pi (\mathbf{j})
 \Longleftrightarrow
I_{\mathbf{i}}=I_{\mathbf{j}}.
\end{equation}
 We are given a probability vector $\mathbf{p}:=(p_0,p_1,p_3)$. That is,
$p_i>0$ and $\sum _{i\in\mathcal{A}}p_i=1$.

\begin{figure}[h!]
    \centering
    \includegraphics[width=\textwidth]{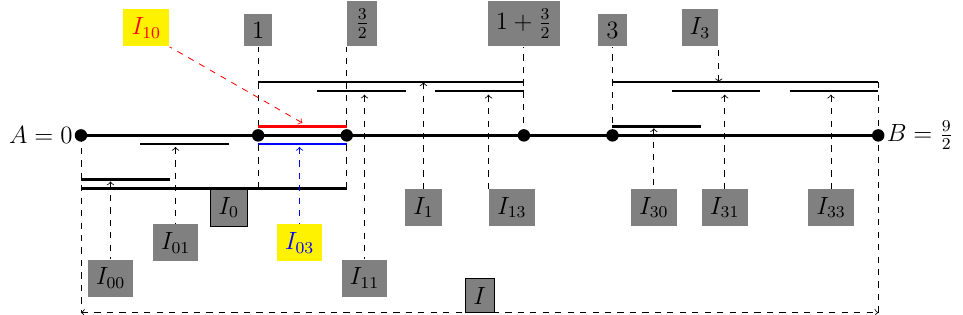}
    \caption{\footnotesize{$I_{i_1\dots  i_n}:=S_{i_1\dots  i_n}(I)$, $(i_1,\dots  ,i_n)\in\mathcal{A}^n$, and
     $I_{10}=I_{03}$.}}\label{y95}
  \end{figure}

We form the corresponding product measure $\mu :=\mathbf{p}^{\mathbb{N}}$ and define its push forward measure $\nu :=\Pi _*\mu $. Then,
$\nu $ is the \texttt{self-similar measure} corresponding to the probability vector $\mathbf{p}$. That is,
for every maximal finite antichain $\Gamma \subset \Sigma ^*$ we have
\begin{equation}
\label{y92}
\nu =\sum_{\mathbf{i}\in\Gamma }
p_{\mathbf{i}}\cdot\nu\circ S _{\mathbf{ i}}^{-1 }.
\end{equation}
The peculiarity of this IFS $\mathcal{S}$ is that we have complete overlap:
$
I_{0,3}=I_{1,0}.
$
Moreover, an easy case analysis yields:
\begin{fact}\label{w92}
  Assume that $\mathbf{i},\mathbf{j}\in \mathcal{A}^n$, with
  $i_1< j_1$ such that
$$
\Pi (\mathbf{i})=\Pi (\mathbf{j})\quad \text{\emph{but}}\quad
\Pi (\mathbf{i}|_k)\ne\Pi (\mathbf{j}|_k),\quad
\forall k<n.
$$
Then, $\mathbf{i}=(\underbrace{1,\dots  ,1 }_{n-1},0)$
and $\mathbf{j}=(0,\underbrace{3,\dots  ,3 }_{n-1})$.

\end{fact}

\begin{definition}
  Let
\begin{equation}\label{xx69}
  A:=\left(
     \begin{array}{ccc}
       1 & 1 & 0 \\
       1 & 1 & 1\\
       1 & 1 & 1 \\
     \end{array}
   \right)
  \text{ and }
   B:=\left(
     \begin{array}{ccc}
       1 & 1 & 1 \\
       0 & 1 & 1\\
       1 & 1 & 1 \\
     \end{array}
   \right),
\end{equation}
where we index the rows and columns with ${0,1,3}$
in increasing order.
We introduce the subshifts of finite types
$$
\Sigma _A:=\left\{ \mathbf{i}\in \Sigma :
 (i_k,i_{k+1})\ne (0,3), \forall k\right\},\
 \Sigma _B:=\left\{ \mathbf{i}\in \Sigma :
 (i_k,i_{k+1})\ne (1,0), \forall k\right\}.
$$
Analogously, we define
$$
\mathcal{T}_n:=\left\{ \mathbf{i}\in\mathcal{A}^n:
(i_k,i_{k+1})\ne (0,3), \forall k<n
\right\},
$$ and
$\mathcal{U}_n:=\left\{ \mathbf{i}\in\mathcal{A}^n:
(i_k,i_{k+1})\ne (1,0), \forall k<n
\right\}.
$
Set
$$
\Sigma _A^*:=\bigcup _{n=1}^{\infty   }\mathcal{T}_n\cup \flat \text{ and }
\Sigma _B^*:=\bigcup _{n=1}^{\infty   }\mathcal{U}_n\cup \flat,
$$
where $\flat$ is the empty word.

\end{definition}

An obvious case analysis shows (see
\cite[Fact 4.2.18]{BSSbook}) that the following fact holds:
\begin{fact}\label{y90}
  \begin{enumerate}[{\bf (a)}]
\item Assume that for the distinct $\mathbf{i},\mathbf{j}\in \mathcal{T}_n$ we have $I_{\mathbf{i}}\cap I_{\mathbf{j}}\ne \emptyset $. Then, $|\mathbf{i}\wedge \mathbf{j}|=n-1$ and $\left\{ i_n,j_n \right\}=\left\{ 0,1 \right\}$, where $\mathbf{i}\wedge \mathbf{j}$ is the common prefix of the words $\mathbf{i}$ and $\mathbf{j}$.
\item Assume that for the distinct $\mathbf{i},\mathbf{j}\in \mathcal{U}_n$ we have $I_{\mathbf{i}}\cap I_{\mathbf{j}}\ne \emptyset $. Then,
there exists a $k\leq n-2$ and an $\omega \in \mathcal{T}_k$ such that $\omega _{k}\ne 1$,
 $\mathbf{i}=\omega 0 \overline{3}^{n-k-1}$ and
 $\mathbf{j}=\omega \overline{1}^{n-k} $.
  \end{enumerate}
\end{fact}
For an $\mathbf{i}\in\mathcal{T}_n$ there can be exponentially
many $\mathbf{j}\in\mathcal{A}^n$ with $I_{\mathbf{i}}\cap I_{\mathbf{j}}\ne \emptyset $. However,
the previous
 Fact implies the following corollary.
\begin{corollary}\label{y89}
 If  $\mathbf{i}\in \mathcal{T}_n$,  then there is at most one $\mathbf{j}\in\mathcal{T}_n\setminus\left\{ \mathbf{i} \right\}$ such that
 $I_{\mathbf{i}}\cap I_{\mathbf{j}}\ne \emptyset $.
 The same remains valid if we replace $\mathcal{T}_n$
 with $\mathcal{U}_n$.
\end{corollary}
If $p_1\geq p_3$, then we should work with
$\Sigma _A$ and $\mathcal{T}_n$, $n\in\mathbb{N}$. On the other hand, if
$p_3> p_1$, then we should work with
$\Sigma _B$, and $\mathcal{U}_n$, $n\in\mathbb{N}$.
\begin{PA}\label{y88}
  We always assume in this note that
\begin{equation}
 \label{y87}
 p_3\leq p_1.
 \end{equation}
\end{PA}
For the symmetry pointed out in Corollary \ref{y89}, we may assume without any loss of generality for the rest of the paper that \eqref{y87} holds.

 It is immediate from part (a) of Fact \ref{y90} that
\begin{equation}
\label{y86}
\mathbf{i},\mathbf{j}\in\mathcal{T}_n,\qquad
\mathbf{i}=
\mathbf{j} \Longleftrightarrow
S_{\mathbf{i}}= S_{\mathbf{j}}
\Longleftrightarrow
\Pi (\mathbf{i})= \Pi (\mathbf{j}).
\end{equation}

\section{The main result}

\begin{definition}[Pressure of a potential]\label{u95}
  We call a continuous function
   $f: \Sigma^* _A\to [0,\infty  )$ a potential.
   The pressure of the potential $f$
 is defined by
\begin{equation}
\label{u94}
P(f):=
\lim\limits_{n\to\infty}
\frac{1}{n}\log \sum _{\mathbf{ i}\in \mathcal{T}_n }
f (\mathbf{i}),
\end{equation}
if the limit exists, otherwise we replace the limit with limsup.
\end{definition}

\subsection{The main result}
As we mentioned above we assume that $p_1\geq p_3$. If $p_1\leq
p_3$, then all $\mathcal{T}_n$ below should be replaced by $\mathcal{U}_n$ and all the results remain unchanged.
We define
\begin{equation}
\label{u99}
\mathcal{I}_{\mathbf{i}}:=\left\{ \pmb{\eta}\in
\mathcal{A}^{n}:
S_{\pmb{\eta}}=S_{\mathbf{i}}
\right\}\quad \text{ and }\quad
\psi (\mathbf{i}):=\sum_{\pmb{\eta}\in \mathcal{I}_{\mathbf{i}}} p_{\pmb{\eta}},\quad\text{ for every }
\mathbf{i}\in\Sigma ^*_A,
\end{equation}
and we define
\begin{equation}
\label{u39}
\psi(\flat):=1, \text{ where } \flat \text{ is the empty word. }
\end{equation}
We will point out in \eqref{c97} that
\begin{equation}
\label{c99}
\#\mathcal{I}_n\leq n.
\end{equation}
We will prove in Section \ref{c98} that the limit in the following definition exists:
\begin{equation}
\label{u98}
p(t):=P(\psi ^t)=
\lim\limits_{n\to\infty}\frac{1}{n}\log \sum_{\mathbf{ i}\in \mathcal{T}_n}\psi^t (\mathbf{i}),\quad
t\geq 0.
\end{equation}
It follows from \eqref{c99} that for every $t\geq 0$ and $\varepsilon >0$
\begin{equation}
\label{c96}
\frac{p(t+\varepsilon )-p(t)}{\varepsilon }
\in
\left(
  \log p_{\min},\log p_{\max}
 \right).
\end{equation}
Using this and the definition of $\psi (\mathbf{i})$ we get that
the function $p(t)$ has the following properties:\label{c92}
\begin{enumerate}
[{\bf (i)}]
\item $p(t)$ is a continuous and strictly decreasing function on $[0,\infty  )$,
  \item $p(1)=0$ and so $p(t)>0$ for $t\in[0,1)$.
\end{enumerate}
This implies that  there exists a unique $t_r\in(0,1)$ such that
\begin{equation}
\label{u90}
p(t_r):=
r t_r\log 3, \quad \text{ for every } r\in(0,\infty  ).
\end{equation}
Observe that by the properties (i) and (ii) of $p(t)$ we obtain that
\begin{equation}
\label{c94}
t_r>0,\quad \text{ for all } r>0.
\end{equation}
We define
$\chi _r$ such that
\begin{equation}
\label{x71}
t_r=\frac{\chi _r}{r+\chi _r},
\text{ that is, }\quad
\chi _r=\frac{t_r r}{1-t_r}.
\end{equation}

\begin{theorem}\label{x70} The quantization dimension of the measure $\nu $ is
  $D_r(\nu )=\chi _r$.
\end{theorem}

\subsection{An explicit formula for the $L^q$-spectrum of $\nu $}
As we mentioned, the combination of \cite[Corolary 1.12]{KNZ} and Theorem \ref{x70} yields an explicit formula for the $L^q$-spectrum of $\nu $
for $q\in(0,1)$. Namely,
the $L^q$-spectrum of the measure $\nu $ is defined for $q>0$ by
\begin{equation}
\label{r99}
\beta (q):=\limsup\limits_{n\to\infty  }
\frac{\log \sum\limits_{C\in\mathcal{D}_n}\nu (C)^q}{\log 2^n},
\end{equation}
where $\mathcal{D}_n:=\left\{(k/2^n,(k+1)/2^n] :
k\in\mathbb{Z}
\right\}$ is the partition of $\mathbb{R}$
into dyadic intervals. In \cite{KNZ} the authors introduced
\begin{equation}
\label{r93}
q_r:=\inf\left\{ q>0:\beta (q)<qr \right\}.
\end{equation}
In our special case,  \cite[Corolary 1.12]{KNZ} yields
\begin{equation}
\label{r96}
D_r(\nu )=\frac{r q_r}{1-q_r}, \quad \text{ for every } r>0.
\end{equation}
Putting together \eqref{x71}, the assertion of Theorem \ref{x70} and \eqref{r96} we get
\begin{equation}
\label{r95}
\frac{r t_r}{1-t_r}=
D_r(\nu )=\frac{r q_r}{1-q_r},  \quad \text{ for every } r>0.
\end{equation}
The combination of this and \eqref{c94} yields
\begin{equation}
\label{c95}
q_r=t_r>0.
\end{equation}
We know that the function $\beta (q)$ convex and in this way continuous on $(0,\infty  )$. Hence, by the definition of $q_r$ we get
\begin{equation}
\label{r94}
\beta (q_r)=rq_r.
\end{equation}

\begin{corollary}\label{r97}
For every $q\in(0,1)$ we have
\begin{equation}
\label{r98}
\beta (q)=\frac{p(q)}{\log 3}.
\end{equation}
\end{corollary}
\begin{proof}[Proof of Corollary \ref{r97} assuming Theorem \ref{x70}]
 Putting together \eqref{u90}, \eqref{c95} and \eqref{r94} we get
\begin{equation}
\label{c93}
\frac{p(t_r)}{r\log 3}=t_r=q_r=
\frac{\beta (q_r)}{r}, \quad \text{for all } r>0.
\end{equation}
Hence,
\begin{equation}
\label{r92}
p(q_r)=  \beta (q_r)\cdot\log 3,  \quad \text{ for every } r>0.
\end{equation}
Choose an arbitrary $q'\in (0,1)$.
Let $r:=p(q')/q'$.
Observe that   $r>0$ by property (ii) of the function $p(t)$. Then,
$q_r=q'$. Hence, by \eqref{c93} we get that
$\beta (q')=\frac{p(q')}{\log 3}$.
\end{proof}

\subsection{The main Proposition}\label{w99}
To prove our main result we need our Main Proposition
(Proposition \ref{u18}) below. To state it we need some further notation.

\subsubsection{ A projection $\Phi : \Sigma\cup \Sigma ^* \to \Sigma_A\cup \Sigma_A^*$}
Recall that by definition $\Sigma\cup \Sigma ^*$ is the collection of finite or infinite words over the alphabet $\mathcal{A}$. Similarly, $\Sigma_A\cup \Sigma_A^*\subset \Sigma\cup \Sigma ^*$
is the collection of those elements of $\Sigma\cup \Sigma ^*$, which do not contain the sequence $(0,3)$.
First we define such a mapping $\Phi: \Sigma\cup \Sigma ^* \to \Sigma_A\cup \Sigma_A^* $ which has the following properties:
For every $\pmb{\eta}\in \Sigma\cup \Sigma ^*$
\begin{enumerate}
[{\bf (a)}]
  \item $\Phi $ preserves the length of every word: $|\pmb{\eta}|=|\Phi (\pmb{\eta})|$,
  \item $\Pi (\pmb{\eta})=\Pi (\Phi (\pmb{\eta}))$.
\end{enumerate}
In the rest of the paper we frequently use the following notation:
 For a digit $a$ and $n\in \mathbb{N}\cup\{\infty  \}$,
$$
\overline{a}^n:=\underbrace{
a,a,\dots  ,a}_n.
$$
\begin{definition}\label{w91}
 We define $\Phi :\Sigma\cup \Sigma ^* \to \Sigma_A\cup \Sigma_A^*$ as follows:
Let $\pmb{\eta }\in \Sigma $.
We obtain $\mathbf{i}:=(i_1,i_2,\dots  ):=\Phi (\pmb{\eta })\in \Sigma_A$ from $\pmb{\eta }\in \Sigma $ by successive substitutions as follows
\begin{enumerate}
[{\bf (a)}]
  \item For every $1\leq k<\ell <\infty  $ such that $(\eta _k,\dots  ,\eta _{\ell })=(0,\overline{3}^{\ell-k })$ and $\eta _{\ell +1}\ne 3$,  we define
  $(i_k,\dots  ,i_{\ell }):=(\overline{1}^{\ell-k },0)$.
\item If there exists a $1\leq k<\infty  $ such that
$(\eta _k, \eta_{k+1},  \eta_{k+1},\dots  )=(0,\overline{3}^{\infty  })$, then we define
$(i_k,i_{k+1},\dots  ):= (\overline{1}^{\infty  })$.
\end{enumerate}
\end{definition}
Then, we get rid of all the $(\eta _k,\eta _{k+1})=(0,3)$ in $\pmb{\eta }$, so
$\Phi (\pmb{\eta })\in \Sigma _A\cup \Sigma_A^*$ and
by Fact \ref{w92}:
\begin{equation}
\label{w89}
\Pi (\Phi (\pmb{\eta} ))=\Pi (\pmb{\eta }),\qquad
\forall \pmb{\eta }\in \Sigma \cup\Sigma ^*.
\end{equation}
However,
considerable technical difficulties are caused by the fact that
\begin{equation}
\label{w88}
\Phi \circ \sigma \ne \sigma \circ \Phi.
\end{equation}
Namely, $\sigma \Phi (0,3)=\sigma  10=0\ne 3=\Phi (3)= \Phi (\sigma 03)$.

Moreover,
\begin{equation}
\label{y08}
\exists\, \pmb{\eta },\text{ and } \exists\, m<|\pmb{\eta }|,
\qquad
\Phi (\pmb{\eta}|_m)\ne \Phi (\pmb{\eta})|_m.
\end{equation}
For example,
\begin{equation*}
\label{y07}
\Phi (0,0,3)=(0,1,0), \text{ but }
\Phi ((0,0,3)|_2)=\Phi (0,0)=(0,0)\ne
(0,1)=\Phi (0,0,3)|_2.
\end{equation*}

D.J. Feng \cite{feng2011equilibrium} introduced
a very important family of potentials.
This family was
termed quasi-multiplicative potentials by
A. K\"aenm\"aki and H.W. Reeve \cite{kaenmaki2014multifractal}).

\subsubsection{Weak quasi-multiplicative potentials}

\begin{definition}\label{u19}
  We say that a function  $\phi: \Sigma _{A}^{*}\to [0,\infty  ) $ is a
  \texttt{weak quasi-}
  \texttt{multiplicative potential} on $\Sigma^* _A$
if the following two conditions hold:
\begin{enumerate}[{\bf (a)}]
\item
There is an $\pmb{\ell }\in \Sigma _A^*$ which is not the empty word such that $\phi (\pmb{\ell })>0$. Moreover,
there exist
$C_1,C_2>0$ such that
\begin{equation}
\label{y85}
\phi (\mathbf{i}\mathbf{j})\leq C_1
\phi (\mathbf{i})\phi (\mathbf{j}),\qquad
\mathbf{i}\mathbf{j}\in\Sigma _A^*.
\end{equation}
\item
There exists a $z\in \mathbb{N}$ such that
\begin{equation}
\label{y84}
\forall \mathbf{i},\mathbf{j}\in \Sigma _A^*,
\exists \mathbf{k}\in \bigcup _{\ell =1}^{z }
\mathcal{T}_{\ell }\cup \flat
\text{ such that }
\mathbf{i}\mathbf{k}\mathbf{j}\in \Sigma _A^*
\text{ and }
\phi (\mathbf{i})\phi(\mathbf{j})\leq
C_2\phi (\mathbf{i}\mathbf{k}\mathbf{j}).
\end{equation}
\end{enumerate}
\end{definition}

First we introduce a potential $\widehat{\psi }:\Sigma^* _A \to (0,1]$ as follows:

\begin{equation}
  \label{a91}
  \widehat{\psi}(\mathbf{i}):
  =\left\{
  \begin{array}{ll}
    \max\left\{
      \psi(\mathbf{i}),\psi(\mathbf{i}^-0)
       \right\}
  ,&
  \hbox{if $\mathbf{i}_{|\mathbf{i}|}=1$;}
  \\
  \psi(\mathbf{i})
  ,&
  \hbox{if $\mathbf{i}_{|\mathbf{i}|}\ne 1$,}
  \end{array}
  \right. \quad
  \text{ for } \mathbf{i}\in\Sigma^* _A.
  \end{equation}

\subsubsection{The statement of the  Main Proposition}

\begin{proposition}\label{u18}\
  The following properties hold:
  \begin{enumerate}[{\bf Property}-\bf 1]
\item\label{w98} $\widehat{\psi }$ is a weak quasi-multiplicative potential on $\Sigma^* _A$.
\item\label{w97}
There exists a $C_3>0$ such that
for every $n\geq 1$
\begin{equation}
\label{u15}
1\leq \frac{\widehat{\psi}(\mathbf{i})}{\psi(\mathbf{i})}\leq C_3\cdot n,\qquad
\forall  \mathbf{i}\in \mathcal{T}_n.
\end{equation}
\item\label{w96}  For every $n\geq 1$
\begin{equation}
\label{w95}
\sum_{ \mathbf{i}\in\mathcal{T}_n}
\ind_{\mathcal{I}_{\mathbf{i}}}(x)\leq 2.
\end{equation}
\item\label{w93}
 For every $n\in \mathbb{N}\cup\{\infty  \}$, $\pmb{\eta }\in \mathcal{A}^n$, $1\leq z<n$,  $\mathbf{i}:=\Phi (\pmb{\eta })|_z$

 \begin{equation}
\label{v99}
\pmb{\eta }|_n\subset \mathcal{I}_{\mathbf{i}}\cup \mathcal{I}_{\mathbf{i}^-0}.
\end{equation}

\end{enumerate}
\end{proposition}

Property-\ref{w97} is an immediate consequence of Fact \ref{u83},

 Property-\ref{w96} is an immediate consequence of Part (a) of Fact \ref{y90},

Property-\ref{w93} is proved  in Part (c) of Fact \ref{y13}.

\textbf{The organization of the rest of the paper}
\begin{enumerate}[{\bf (a)}]
  \item In Section \ref{w94} we introduce further pressure functions and finite maximal antichains. Moreover, we prove some of their properties, assuming Proposition \ref{u18}.
  \item In Sections \ref{y16} and \ref{a99}
  we prove Theorem \ref{x70} using only those properties of $\widehat{\psi}(\mathbf{i})$ which are listed in Proposition \ref{u18}.
  \item In Section \ref{aa99} we prove that $\widehat{\psi}(\mathbf{i})$
  satisfies the Properties-\ref{w98},\ref{w97} listed in Proposition \ref{u18}.
\end{enumerate}

\section{Pressure functions and finite maximal antichains}\label{w94}
We will use the following theorem due to D.J. Feng
\cite[Theorem 5.5]{feng2011equilibrium}.
\begin{theorem}[Feng]\label{y36}
Let $\phi $ be a weak quasi-multiplicative potential on $\Sigma^* _A$. Then, there exists a unique invariant ergodic measure $ \mathfrak{m} $ on $\Sigma _A$ with the following property
\begin{equation}
\label{y83}
\mathfrak{m}(\mathbf{i})
\thickapprox
\frac{\phi (\mathbf{i})}{\sum _{\mathbf{ j}\in \mathcal{T}_n}\phi (\mathbf{j})}
\thickapprox
\phi (\mathbf{i})\exp\left(-n P(\phi )  \right),
\end{equation}
where
$ a(\mathbf{i})\thickapprox b(\mathbf{i})$ if there exists a $c>0$ such that
$\frac{1}{c}b(\mathbf{i})\leq
a(\mathbf{i})\leq c b(\mathbf{i})$
for all $\mathbf{i}\in\Sigma _A^*$ and
\begin{equation}
\label{u17}
P(\phi )=\lim\limits_{n\to\infty}\log \sum\limits_{\mathbf{i}\in\mathcal{T}_n}
\phi (\mathbf{i}).
\end{equation}
\end{theorem}

\subsection{Various pressure functions}\label{c98}

By Theorem \ref{y36}, we obtain that
\begin{equation}
\label{u16}
\text{ the limit }
\quad
\lim\limits_{n\to\infty}\frac{1}{n}\log \sum_{\mathbf{ i}\in\mathcal{T}_n}\left( \widehat{\psi}(\mathbf{i}) \right)^t
\quad
\text{exists.}
\end{equation}
Recall that the pressure function $p(t)=P(\psi ^t)$
was defined in \eqref{u98} with the comment that the existence of the limit in \eqref{u98} would be proved later.
Using Property-\ref{w97} and \eqref{u16} we get that
the second equation below holds:
\begin{equation}
\label{u14}
p(t)=\lim\limits_{n\to\infty}\frac{1}{n}\log \sum_{\mathbf{ i}\in\mathcal{T}_n}\left( \psi(\mathbf{i}) \right)^t
=
\lim\limits_{n\to\infty}\frac{1}{n}\log \sum_{\mathbf{ i}\in\mathcal{T}_n}\left( \widehat{\psi}(\mathbf{i}) \right)^t.
\end{equation}
In particular, the first limit (which was defined as $p(t)$ in~\eqref{u98}) exists.

For every $t>0$ we define the potentials
\begin{equation}
\label{y39}
\phi _t(\mathbf{i}):=
\left( \psi(\mathbf{i})\cdot 3^{-|\mathbf{i}|r}
\right)^t  \quad
\text{ and } \quad
\widehat{\phi }_t(\mathbf{i}):=
\left( \widehat{\psi}(\mathbf{i})\cdot 3^{-|\mathbf{i}|r}
\right)^t.
\end{equation}
We obtain from Property-\ref{w98} that
\begin{corollary}\label{y38}
  For every $t>0$ the potential
$ \mathbf{i}\mapsto \widehat{\phi }_t(\mathbf{i}) $
is also quasi-multiplicative.
\end{corollary}
We cannot say the same about the potential  $\phi _t$
but by \eqref{u14} the pressure functions of
$
\phi _t
$
and
$
\widehat{\phi }_t
$ are the same:
\begin{equation}
\label{u13}
P(t):=\lim\limits_{n\to\infty}
\frac{1}{n}\log \sum_{\mathbf{j}\in\mathcal{T}_n }
\phi _t(\mathbf{j})
=\lim\limits_{n\to\infty}
\frac{1}{n}\log \sum_{\mathbf{j}\in\mathcal{T}_n }
\widehat{\phi }_t(\mathbf{j})
=p(t)-rt\log 3.
\end{equation}
Using the properties of the function $p(t)$ stated on page \pageref{c92}
we obtain that
$t\mapsto P(t)$ is also strictly decreasing,
$P(0)=p(0)>0  $
and $P(1)=-r\log 3$. So, we obtain that
\begin{fact}\label{W87}
There is a unique
$t_r\in(0,1)$ such that
\begin{equation}
\label{y35}
p(t_r)=rt_0\log 3,\quad
P(t_r)=0, \quad P(t)>0 \text{ if }
t\in [0,t_r) \text{ and }
P(t)<0 \text{ if } t>t_r.
\end{equation}
\end{fact}
Finally, we introduce the potential $\widehat{\phi} :\Sigma ^*_A\to [0,\infty  )$,
\begin{equation}
\label{y14}
\widehat{\phi}(\mathbf{i})
:=\widehat{\phi }_{t_r}(\mathbf{i})=\left( \widehat{\psi}(\mathbf{i})\cdot 3^{-|\mathbf{i}|r}
\right)^{t_r}.
\end{equation}
Then, by definition, the pressure of $\widehat{\phi} $ is equal to $0$:
$$
P(\widehat{\phi} )=\lim\limits_{n\to\infty}\log \sum\limits_{\mathbf{i}\in\mathcal{T}_n}
\widehat{\phi} (\mathbf{i})
=0.
$$ Moreover, as a corollary of Feng Theorem  (Theorem \ref{y36}) and Corollary \ref{y38} we obtain:
\begin{proposition}\label{y34}
  There is a $C_4>1$ and a unique invariant ergodic measure
  $\mathfrak{m}$ on $\Sigma _A$ such that
  \begin{equation}
  \label{y33}
  C _{4}^{-1 }<
\frac{\mathfrak{m}([\mathbf{i}])}{\widehat{\phi} (\mathbf{i})}
  <C_4,\quad \text{ for all }\quad \mathbf{i}\in\Sigma _A^*.
  \end{equation}
\end{proposition}

\subsection{Finite maximal antichains of $\Sigma_A$ and $\Sigma $}
Let $X$ be either $\Sigma _A$ or $\Sigma $.
A finite collection $\Gamma $ of finite words  $\mathbf{i}$ of $X$ is a
\texttt{finite maximal antichain} of $X$ if for every $\pmb{\omega}\in X$
we can find a unique $\mathbf{i}\in \Gamma $ such that $\pmb{\omega}\in [\mathbf{i}]$.

\begin{definition}\label{u11}
Let $\widetilde{q}_i:=\left( p_i3^{-r} \right)^{t_r}$ and
$\varepsilon _0:=\min\limits_{i\in\mathcal{A}}\widetilde{q}_i$. For an $0<\varepsilon <\varepsilon _0$ we define
\begin{equation}
\label{u10}
\widehat{\Gamma} (\varepsilon )\!:\!=
\Big\{
  \mathbf{i}\in\Sigma ^*_A:
  \widehat{\phi} (\mathbf{i})<\varepsilon ,\
  \forall p<|\mathbf{i}|,\
  \widehat{\phi} (\mathbf{i}|_p)\geq \varepsilon
\Big\} \text{ and }
 \widehat{\Gamma}^- (\varepsilon ):=
 \Big\{ \mathbf{i}^-\!
   :\!\mathbf{i}\in \widehat{\Gamma} (\varepsilon)
 \Big\}.
\end{equation}
\end{definition}
It is clear that both  $\widehat{\Gamma} (\varepsilon )$
and $\widehat{\Gamma}^- (\varepsilon )$
are  maximal antichains of $\Sigma _{A}$. From this and from
\eqref{y33}, we get that
\begin{equation}
\label{a88}
\sum_{\mathbf{ i}\in \widehat{\Gamma}(\varepsilon)}
\widehat{\phi}(\mathbf{i})\leq C_4.
\end{equation}

\begin{lemma}\label{u12}
  There exists a $\gamma '>0$ such that for all
  $0<\varepsilon <\varepsilon _0$, we have
  \begin{equation}
  \label{u09}
  \frac{\sum\limits_{\mathbf{i}\in\widehat{\Gamma }(\varepsilon ) }
 \widehat{\phi}(\mathbf{i}) }
  {\sum\limits_{\mathbf{i}\in\widehat{\Gamma }(\varepsilon ) }
  \widehat{\phi}(\mathbf{i}^-)}>\gamma '.
  \end{equation}
\end{lemma}
 \begin{proof}
  Observe that
\begin{equation}
\label{u08}
\sum\limits_{\mathbf{i}\in\widehat{\Gamma }(\varepsilon ) }
  \widehat{\phi}(\mathbf{i}^-) \leq
  3
  \sum\limits_{\mathbf{j}\in\widehat{\Gamma }^-(\varepsilon ) }
  \widehat{\phi}(\mathbf{j})\leq
  3C_4 \sum\limits_{\mathbf{j}\in\widehat{\Gamma }^-(\varepsilon ) }
  \mathfrak{m}([\mathbf{i}])\leq 3C_4.
\end{equation}
Namely, the first inequality follows from the fact that
for every $\mathbf{j}\in\widehat{\Gamma }^-(\varepsilon )$
 there are at most three $\mathbf{i}\in\widehat{\Gamma }(\varepsilon )$
such that $\mathbf{j}=\mathbf{i}^-$. The second inequality is immediate from \eqref{y33}. The third inequality is a consequence of the fact that $\widehat{\Gamma}^- (\varepsilon )$ is a maximal antichain of $\Sigma _A$. On  the other hand,
\begin{equation}
\label{u07}
\sum\limits_{\mathbf{i}\in\widehat{\Gamma }(\varepsilon ) }
 \widehat{\phi}(\mathbf{i})
 \geq
 C _{4}^{-1}
 \sum\limits_{\mathbf{i}\in\widehat{\Gamma }(\varepsilon ) }
 \mathfrak{m}(\mathbf{i})=C _{4}^{-1}.
\end{equation}
  Putting together \eqref{u08} and \eqref{u07} we obtain that
  \eqref{u09} holds with the choice of $\gamma ':=\frac{1}{3C _{4}^{2}}$.
 \end{proof}
 \begin{corollary}\label{u06} Let $0<\varepsilon <\varepsilon _0$. Then, we have
\begin{equation}
\label{u05}
\# \widehat{\Gamma }(\varepsilon )\leq
  \frac{3C_4^3}{\varepsilon }.
\end{equation}
 \end{corollary}
 \begin{proof}
   It follows from \eqref{y33} that
   $C_4> \sum\limits_{\mathbf{i}\in\widehat{\Gamma }(\varepsilon ) }
   \widehat{\phi}(\mathbf{i})$. Putting together this, \eqref{u09},
   and  the fact that   $\widehat{\phi}(\mathbf{i}^-)\geq \varepsilon $,
   we obtain
   $$
C_4>\gamma ' \sum\limits_{\mathbf{i}\in\widehat{\Gamma }(\varepsilon ) }
\widehat{\phi}(\mathbf{i}^-)
>
\gamma '\varepsilon\# \widehat{\Gamma }(\varepsilon ) .
   $$
   This implies that \eqref{u05} holds since $\gamma '=\frac{1}{3C _{4}^{2}}$.
 \end{proof}
 Although $\widehat{\Gamma }(\varepsilon )$ is a finite maximal antichain the same is not true (in general) for
 $
\bigcup\limits_{\mathbf{i}\in \widehat{\Gamma }(\varepsilon ) }
\mathcal{I}_{\mathbf{i}}
 $, where we defined $\mathcal{I}_{\mathbf{i}}$ in \eqref{u99}.
So, we need to introduce one more step:
\begin{definition}\label{a98}
\begin{equation}
\label{u02}
\Theta(\varepsilon ):
 =
 \left\{ \mathbf{i}^-0:
 \mathbf{i}\in\widehat{\Gamma} (\varepsilon ) \text{ and }
  i_{|\mathbf{i}|}=1
  \right\}\subset
  \Sigma _{A}^{*}.
\end{equation}
Moreover, let
\begin{equation}
\label{x99}
\widehat{\Gamma} _E(\varepsilon ):=
\widehat{\Gamma} (\varepsilon )\cup \Theta(\varepsilon )
\quad
\text{ and }
\quad
\widehat{\Gamma}_{\Sigma}(\varepsilon):=
\bigcup\limits_{\mathbf{j}\in  \widehat{\Gamma}_E(\varepsilon) }
\mathcal{I}_{\mathbf{j}}.
\end{equation}
\end{definition}

The following claim states a simple but important property.
\begin{claim}\label{a83}
  For all $\mathbf{i}\in \widehat{\Gamma}_E(\varepsilon)$
  we have $\widehat{\phi}(\mathbf{i})<\varepsilon $.
\end{claim}
\begin{proof}
  We need to check this only for $\mathbf{i}\in \Theta(\varepsilon )$ since by definition $\widehat{\phi}(\mathbf{i})<\varepsilon $ holds for all $\mathbf{i}\in \widehat{\Gamma}(\varepsilon)$. If $\mathbf{i}\in \Theta(\varepsilon)$, then $\mathbf{i}=\mathbf{i}^-0$ and $\mathbf{i}^-1\in \widehat{\Gamma}(\varepsilon)$. By \eqref{a91},
  $$
  \widehat{\psi}(\mathbf{i})=
  \psi(\mathbf{i})=\psi(\mathbf{i}^-0)\leq
  \max\{\psi(\mathbf{i}^-0),\psi(\mathbf{i}^-1)\}=\widehat{\psi}(\mathbf{i}^-1).
  $$
  So, we get that
\begin{equation}
\label{a82}
\forall \mathbf{i}\in \Theta(\varepsilon ),\quad
\widehat{\phi}(\mathbf{i})= \left(\widehat{\psi}(\mathbf{i})
\cdot 3^{-|\mathbf{i}|r}  \right)^{t_r}
\leq
\left(\widehat{\psi}(\mathbf{i}^-1)
\cdot 3^{-|\mathbf{i}|r}  \right)^{t_r}
=
\widehat{\phi}(\mathbf{i}^-1)<\varepsilon,
\end{equation}
  where in the last step we used that $\mathbf{i}^-1\in \widehat{\Gamma}(\varepsilon)$.
\end{proof}
We will use the following immediate consequence of Claim \ref{a83}:
\begin{equation}
\label{a81}
\left(\psi(\mathbf{i}) 3^{-|\mathbf{i}|r}  \right)^{t_r}
\leq
\left(\widehat{\psi}(\mathbf{i}) 3^{-|\mathbf{i}|r}  \right)^{t_r}
=
\widehat{\phi}(\mathbf{i})<\varepsilon
\text{ holds for all }
 \mathbf{i}\in \widehat{\Gamma}_E(\varepsilon).
\end{equation}

\begin{lemma}\label{a99}
For every $0<\varepsilon <\varepsilon _0$, we have
\begin{equation}
\label{a95}
\Sigma = \bigcup\limits_{\pmb{\tau}\in \widehat{\Gamma}_{\Sigma}(\varepsilon)}
[ \pmb{\tau} ].
\end{equation}
\end{lemma}

\begin{proof}
 Let $\pmb{\eta}\in \Sigma $. Then, there  is a unique $\mathbf{i}\in
 \widehat{\Gamma}(\varepsilon)$ such that
 $\Phi (\pmb{\eta})\in [\mathbf{i}]$.
  Let $n:=|\mathbf{i}|$.
Then, $\mathbf{i}=\Phi (\pmb{\eta})|_n$. It follows from
Property-\ref{w93} that
$\pmb{\eta}|_n\in \mathcal{I}_{\mathbf{i}}\cup
\mathcal{I}_{\mathbf{i}^-0}$.
\end{proof}
For $\mathbf{i},\mathbf{j}\in \Sigma  _{A}^{ *}$ we say that $\mathbf{i}$ is a proper prefix of $\mathbf{j}$ if $|\mathbf{i}|<|\mathbf{j}|$
and $\mathbf{j}|_{|\mathbf{i}|}=\mathbf{i}$. In this case we write
$\mathbf{i} \precneqq \mathbf{j}$. Let
\begin{equation}
\label{a94}
\widetilde{\Gamma}_{\Sigma}(\varepsilon)
:=
\left\{ \mathbf{j}\in \widehat{\Gamma}_{\Sigma}(\varepsilon) :
\nexists\  \mathbf{i}\in \widehat{\Gamma}_{\Sigma}(\varepsilon) \text{ such that }
\mathbf{i} \precneqq \mathbf{j}
 \right\}.
\end{equation}
The elements of $\widetilde{\Gamma}_{\Sigma}(\varepsilon)$
are incomparable and by \eqref{a95} we have $\Sigma = \bigcup\limits_{\pmb{\tau}\in \widetilde{\Gamma}_{\Sigma}(\varepsilon)}
[ \pmb{\tau} ]$. In this way we have proved that
\begin{fact}\label{a93}
  The collection of cylinders $\widetilde{\Gamma}_{\Sigma}(\varepsilon)$
  is a finite maximal antichain for $\Sigma $. That is, for every $\pmb{\eta}\in \Sigma $ there is a unique $\mathbf{i}\in \widetilde{\Gamma}_{\Sigma}(\varepsilon)$ such that $\pmb{\eta}\in [\mathbf{i}]$.
\end{fact}
\begin{lemma}\label{a92}For an $0<\varepsilon <\varepsilon _0$, we have
\begin{equation}
\label{a90}
\sum_{\mathbf{i}\in \widehat{\Gamma}_E(\varepsilon)} \widehat{\phi}(\mathbf{i})\leq 2C_4.
\end{equation}
\end{lemma}
 \begin{proof}
  Observe that
\begin{equation}
\label{a89}
\sum\limits_{\mathbf{i}\in \widehat{\Gamma}_E(\varepsilon)} \widehat{\phi}(\mathbf{i})\leq\underbrace{
\sum\limits_{\mathbf{i}\in \widehat{\Gamma}(\varepsilon)  }\widehat{\phi}(\mathbf{i})}_{S_1}
+
\underbrace{
\sum\limits_{\mathbf{i}\in \widehat{\Gamma}(\varepsilon) : i|_{|\mathbf{i}|= 1} } \widehat{\phi}(\mathbf{i}^-0)}_{S_2}.
\end{equation}
Putting together formulas \eqref{y39} and \eqref{a91} we get that $S_2\leq S_1$. This and \eqref{a88} together imply that
\eqref{a90} holds.
 \end{proof}

 \begin{definition}\label{y00}
  Using the notation introduced in Definition \ref{u11}
we choose $m,n\in\mathbb{N}$ satisfying
\begin{equation}
\label{y06}
\frac{m}{n}<\varepsilon_0 ^2
\quad \text{ and }\quad
\varepsilon :=6\frac{m}{n}C_4^3.
\end{equation}
Fix such an $m,n$ and $\varepsilon $ for the rest of this section.
\end{definition}
Then, we have from Lemma \ref{a92} and Corollary \ref{u06} that
\begin{equation}
\label{a87}
\# \widehat{\Gamma}_E(\varepsilon)\leq \frac{n}{m}.
\end{equation}
Namely, Corollary \ref{u06} we have $\# \widehat{\Gamma}_E(\varepsilon)\leq 2 \# \widehat{\Gamma}(\varepsilon)\leq \frac{6 C _{4 }^{3 }}{\varepsilon }=\frac{n}{m}$ by the choice of $\varepsilon $.

 \section{The upper estimate for the quantization dimension of the measure $\nu $}\label{y16}
 From now on we follow \cite{GL1}
  and \cite{R3}.
 Recall that $\nu=\Pi _*\mu  $, where $\mu=\mathbf{p}^{\mathbb{N}}$ is the infinite product measure on $\mathcal{A}^{\mathbb{N}}$. Also recall that
\begin{equation}
\label{w87}
V_{n, r}(\nu ):=\inf \left\{\int d(x, \alpha)^r d \nu(x): \alpha \subset \mathbb{R}^{d}, \operatorname{card}(\alpha) \leq n\right\},
\end{equation}
 where $d(x, \alpha)$ denotes the distance from the point $x$ to the set $\alpha$.
 Recall that we introduced
 $C_{n,r}(\mu )$ in \eqref{a85}. It was proved in
\cite{GL1} that $C_{n,r}(\mu )\ne \emptyset $.

In the argument below we will use frequently the potential $\widehat{\phi}(\mathbf{i}) $ which was defined in \eqref{y14},
 the set $\mathcal{I}_{\mathbf{i}}$ which was defined in \eqref{u99},
 the constant  $C_4$ defined in \eqref{y33}.

\subsection{An upper estimate for $V_{n,r}(\nu)$}

First we need to prove a Fact which is a slight modification of \cite[Lemma 4.14]{GL1}.

\begin{fact}\label{a84}
  Let $\mu ,\mu _i$ for $i=1,\dots  ,L$ be Borel probability measures on $\mathbb{R}$
  such that $\int\|x\|^rd\mu _i(x)<\infty  $ for all
  $i=1,\dots  ,L$.
  We are given $q_1,\dots  ,q_L$ positive numbers such that
  \begin{equation}
  \label{x84}
  \mu \leq \sum_{m=1}^{L}q_m\mu _m.
  \end{equation}
  Moreover, we are given natural numbers $\left\{ n_i \right\} _{i=1}^{L}$ such that $n_i\geq 1$ and
$\sum_{i=1}^{L}n_i\leq n$. Then,
\begin{equation}
\label{x83}
V_{n,r}(\mu)\leq
\sum_{i=1}^{L} q_i V_{n_i,r}(\mu_i).
\end{equation}
\end{fact}
\begin{proof}
  Let $\alpha _i\in C_{n_i,r}(\mu _i)$ for $i=1,\dots  ,L$ and $\alpha :=\cup _{i=1}^{L }\alpha _i$.
  \begin{eqnarray}
    V_{n,r}(\mu)  &\leq&
    \int \min_{a\in\alpha }\|x-a\|^rd\mu (x)\leq
    \sum_{i=1}^{L }q_i
    \int \min_{a\in\alpha }\|x-a\|^rd\mu_i (x)
    \\
    &\leq& \sum_{i=1}^{L }q_i
    \int \min_{a\in\alpha_i }\|x-a\|^rd\mu_i (x)
    =
\sum_{i=1}^{L}q_i V_{n_i,r}(\mu_i).
    \end{eqnarray}
\end{proof}

To apply this Fact we prove that the condition \eqref{x84}
holds in our case:

\begin{fact}\label{x81}
  \begin{equation}
  \label{x80}
  \nu \leq \sum_{\mathbf{ i}\in \widehat{\Gamma}_E(\varepsilon) }
  \psi(\mathbf{i})
  \nu _{\mathbf{i}},
  \end{equation}
  where we write
  \begin{equation}
  \label{x79}
  \nu _{\pmb{\eta}}:=(S_{\pmb{\eta}})_*\nu, \text{ for an } \pmb{\eta}\in \mathcal{A}^*.
  \end{equation}
\end{fact}
Note that if $\pmb{\eta}\in \mathcal{I}_{\mathbf{i}}$,
then $S_{\pmb{\eta}}=S_{\mathbf{i}}$. Consequently,
\begin{equation}
\label{x78}
\nu _{\pmb{\eta}}=\nu _{\mathbf{i}},
\text{ for all } \pmb{\eta}\in \mathcal{I}_{\mathbf{i}}.
\end{equation}
\begin{proof} Using that $\widetilde{\Gamma}_{\Sigma }
(\varepsilon )\subset
\widehat{\Gamma}_E(\varepsilon)
$ is a  finite maximal antichain
  \begin{eqnarray*}
    \nu    &=& \sum_{\pmb{\eta}\in \widetilde{\Gamma }_{\Sigma }(\varepsilon )}p_{\pmb{\eta}}\nu _{\pmb{\eta}}
    \leq
    \sum_{\pmb{\eta}\in \widehat{\Gamma}_{\Sigma}(\varepsilon)}p_{\pmb{\eta}}\nu _{\pmb{\eta}}
    =\sum_{\mathbf{ j}\in \widehat{\Gamma}_E(\varepsilon)}
    \sum_{\pmb{\eta}\in \mathcal{I}_{\mathbf{j}}}
    p_{\pmb{\eta}}\nu _{\pmb{\eta}}
    \\
     &=& \sum_{\mathbf{ i}\in \widehat{\Gamma}_E(\varepsilon)}
     \left(\sum_{\pmb{\eta}\in \mathcal{I}_{\mathbf{i}}}
     p_{\pmb{\eta}}\right)\nu _{\mathbf{i}}
     =
     \sum_{\mathbf{ i}\in \widehat{\Gamma}_E(\varepsilon)}
     \psi (\mathbf{i})\nu _{\mathbf{i}}.
    \end{eqnarray*}
\end{proof}

\begin{fact}\label{x77}
  Let us fix an $n_{\mathbf{i}}\in\mathbb{N}$ for every $\mathbf{i}\in\widehat{\Gamma}_E(\varepsilon)$ such that $n_{\mathbf{i}}\geq 1$ and $\sum\limits_{\mathbf{ i}\in \widehat{\Gamma}_E(\varepsilon)}n_{\mathbf{i}}\leq n$.
  Then,
  \begin{equation}
  \label{x76}
  V_{n,r}(\nu )\leq
  \sum_{\mathbf{ i}\in \widehat{\Gamma}_E(\varepsilon) }
  \psi (\mathbf{i})
  V_{n_{\mathbf{i}},r}(\nu_{\mathbf{i}}).
  \end{equation}
\end{fact}
\begin{proof}
  This follows from the combination of Fact \ref{a84} and Fact \ref{x81}.
\end{proof}

\begin{lemma}\label{x82}
  \begin{equation}
    \label{y15}
    V_{n,r}(\nu)\leq
    \inf\left\{
      \sum_{ \mathbf{i}\in \widehat{\Gamma}_E(\varepsilon)  }
     3^{-|\mathbf{i}|r} \psi (\mathbf{i})\cdot
      V_{n_{\mathbf{i}},r}(\nu):
      1\leq n_{\mathbf{i}},\ \sum_{ \mathbf{i}\in \widehat{\Gamma}_E(\varepsilon) }n_{\mathbf{i}}\leq n
    \right\}.
  \end{equation}
\end{lemma}

\begin{proof}
Let $\mathbf{i}\in \widehat{\Gamma}_E(\varepsilon) $,
$n_{\mathbf{i}}\geq 1$, and
$\alpha _{\mathbf{i}}\in C_{n_{\mathbf{i}},r}(\nu )$.
Below we prove that
\begin{equation}
\label{x75}
V_{n_{\mathbf{i}},r}(\nu_{\mathbf{i}})
\leq
3^{-|\mathbf{i}|r}
  V_{n_{\mathbf{i}},r}(\nu).
\end{equation}

 We obtain the assertion of the lemma from the combination of Fact \ref{x77} and \eqref{x75}.
Now we prove \eqref{x75}.
\begin{eqnarray}
\nonumber  V_{n_{\mathbf{i}},r}(\nu_{\mathbf{i}}) &\leq&
  \int d(x, S_{\mathbf{i}}(\alpha _{\mathbf{i}}))^r
  d\nu _{\mathbf{i}}(x)
  =
  \int d(x, S_{\mathbf{i}}(\alpha _{\mathbf{i}}))^r
  d(\nu\circ S_{\mathbf{i}}^{-1})(x)
  \\
  \nonumber    &=&
     \int d(S_{\mathbf{i}}(x), S_{\mathbf{i}}(\alpha _{\mathbf{i}}))^r
  d\nu(x)
  =
  3^{-|\mathbf{i}|r}
  \int d(x, \alpha _{\mathbf{i}})^r
  d\nu(x)\\
      &=&
   3^{-|\mathbf{i}|r}
  V_{n_{\mathbf{i}},r}(\nu).
  \end{eqnarray}
\end{proof}

\begin{proposition}\label{x74}
    There exists a constant $C_9>0$ such that
  \begin{equation}
  \label{x73}
  \limsup\limits_{n\to\infty}
  n \cdot e _{n,r}^{\chi_r }(\nu )
  \leq
  C_9m\cdot
  e _{m,r}^{\chi_r }(\nu )<\infty.
  \end{equation}
  \end{proposition}

  \begin{proof}
    Recall the definition of $n,m,\varepsilon $ from
    Definition \ref{y00}.
    For every $\mathbf{i}\in \widehat{\Gamma}_E(\varepsilon)$
    we define $n_{\mathbf{i}}:=m$.
    Then, by \eqref{a87} we have
    $\sum_{\mathbf{ i}\in \widehat{\Gamma}_E(\varepsilon)}
    n_{\mathbf{i}}=\#\widehat{\Gamma}_E(\varepsilon)m\leq n $.
  We apply Lemma \ref{x82} in the first step below:
    \begin{eqnarray*}
  \nonumber   V_{n,r}(\nu)     &\leq&
    \sum_{\mathbf{ i}\in \widehat{\Gamma}_E(\varepsilon)}
   3^{-|\mathbf{i}|r} \psi (\mathbf{i})V_{m,r}(\nu )
    \\
         &=&\sum_{\mathbf{ i}\in \widehat{\Gamma}_E(\varepsilon)}
  \underbrace{  \left( 3^{-|\mathbf{i}|r}\psi (\mathbf{i}) \right)^{t_r}}_{\leq \widehat{\phi}(\mathbf{i})}
   \underbrace{ \left( 3^{-|\mathbf{i}|r}\psi (\mathbf{i}) \right)^{1-t_r}}_{\leq
   \varepsilon ^{\frac{1-t_r}{t_r}}
   }
    V_{m,r}(\nu )
    \\
    &\leq&
   \underbrace{
    \varepsilon ^{\frac{1-t_r}{t_r}}}_{\varepsilon ^{\frac{r}{\chi _r}}}
   V_{m,r}(\nu )
   \underbrace{\sum_{\mathbf{ i}\in \widehat{\Gamma}_E(\varepsilon)}
   \widehat{\phi} (\mathbf{i})}_{\leq 2C_4}
      \end{eqnarray*}
      The reasoning for the four underbraces above are as follows:
  \begin{enumerate}
  [{\bf (1)}]
    \item This is immediate from the definition $\widehat{\phi} $ and from the fact that $\psi(\mathbf{i})\leq \widehat{\psi}(\mathbf{i})$ for all $\mathbf{i}$.
    \item If $\mathbf{i}\in \widehat{\Gamma}_E(\varepsilon), $
   then $\left( 3^{-|\mathbf{i}r|}\psi (\mathbf{i}) \right)^{t_r}<\varepsilon $
   (see \eqref{a81}).
   \item This follows from definition:  $t_r=\frac{\chi _r}{r+\chi _r}\Longrightarrow \frac{1-t_r}{t_r}=\frac{r}{\chi _r}$.
   \item This follows from Lemma \ref{a92}.
  \end{enumerate}

  So, we have proved that
  \begin{equation}
  \label{x72}
  V_{n,r}(\nu) \leq
  \varepsilon ^{\frac{r}{\chi _r}} V_{m,r}(\nu )
  =
  \left( \frac{m}{n} \right)^{\frac{r}{\chi _r}}
 C_9^{\frac{r}{\chi _r}}
  V_{m,r}(\nu )
  ,
  \end{equation}
  where $C_9:= 6 C _{4}^{3 } $.
  Hence,
  $$
  n V_{n,r}^{\frac{\chi _r}{r}}(\nu)
  \leq
  C_9 m V_{m,r}^{\frac{\chi _r}{r}}(\nu).
  $$
   Letting $n$ approaching to infinity we obtain that \eqref{x73} holds.
  \end{proof}
  Putting together Proposition \eqref{x74}
  and \eqref{y97} we obtain that
  \begin{equation}
  \label{x30}
  \overline{D}_r(\nu )\leq \chi _r.
  \end{equation}

\section{The lower estimate}\label{a99}
First we prove a Fact similar to Fact \ref{x81}.
\begin{fact}\label{x59}
  For every $n$ we have
  \begin{equation}
  \label{x58}
  \nu =\sum_{\mathbf{i}\in \mathcal{T}_n}
  \psi(\mathbf{i}) \nu _{\mathbf{i}}.
  \end{equation}
\end{fact}
\begin{proof}
  Fix an $n$.
  Then, $\mathcal{A}^n$ is a finite maximal antichain for $\Sigma $. Hence,
  \begin{equation}
  \label{x57}
  \nu =\sum_{\pmb{\eta}\in \mathcal{A}^n}
  p_{\pmb{\eta}}\cdot \nu _{\pmb{\eta}},
  \end{equation}
  where
  $
  \nu _{\pmb{\eta}}:=
  \nu \circ S _{\pmb{\eta}}^{-1}
  $.
  Moreover, let $f:\Sigma \to\mathbb{R}$
  be a continuous function.
  Using that $S_{\pmb{\eta}}=S_{\mathbf{i}}$ if $\pmb{\eta}\in \mathcal{I}_{\mathbf{i}}$ we get
  \begin{eqnarray}\label{x53}
   \nonumber  \int f(x)d\nu (x)    &=& \sum_{\pmb{\eta}\in \mathcal{A}^n}
  p_{\pmb{\eta}}
  \int f(x)d\nu_{\pmb{\eta}} (x)=
  \sum_{\mathbf{ i}\in \mathcal{T}_n}
  \sum_{\pmb{\eta}\in \mathcal{I}_{\mathbf{i}}}
  p_{\pmb{\eta}}
  \int f(x)d(\nu \circ S _{\pmb{\eta}}^{-1 })(x) \\
       &=& \sum_{\mathbf{ i}\in \mathcal{T}_n}
      \underbrace{ \sum_{\pmb{\eta}\in \mathcal{I}_{\mathbf{i}}}
       p_{\pmb{\eta}}}_{\psi(\mathbf{i})}
       \int f(x)d(\nu \circ S _{\mathbf{i}}^{-1 })(x)
 =
 \sum_{\mathbf{ i}\in \mathcal{T}_n}
 \psi(\mathbf{i})
 \int f(x)d\nu_{\mathbf{i}}(x).
    \end{eqnarray}
\end{proof}

Let $U:=\text{int}(I)=
\left( 0,\frac{9}{2} \right)$. Recall that $I=\overline{U}\supset \Lambda $. For an $\pmb{\eta}\in \Sigma^*$ we write $U_{\pmb{\eta}}:=S_{\pmb{\eta}}(U)\subset U$. Following Graf and Luschgy
  \cite{GL3} we introduce
\begin{equation}
\label{x69}
u_{n,r}(\nu ):=
\inf \left\{\int d(x, \alpha\cup U^c)^r d \nu(x): \alpha \subset \mathbb{R}^{d}, \operatorname{card}(\alpha) \leq n\right\}.
\end{equation}
As an analogue of \cite[Lemma 4.4]{GL3} we need the following assertion:
\begin{lemma}\label{x65}
  There exists a set $\alpha _n\subset \mathbb{R}^n$
  such that $\# \alpha _n\leq n$ and
  \begin{equation}
  \label{x64}
  u_{n,r}(\nu)=\int d\left( x,\alpha _n\cup U^c \right)^rd\nu (x).
  \end{equation}
\end{lemma}

\begin{proof}
  Define $f:I^n\to \mathbb{R}$ by
$
f(x_1,\dots  ,x_n):=\int
d(
x,{x_1,\dots  ,x_n,0,9/2}
)^rd\nu (x).
$
It is easy to see that this function is continuous. Hence, it attains its infimum on
$I^n$. Any place where the infimum is attained can be chosen as $\alpha _n$.
\end{proof}

\begin{lemma}\label{x67}
 Fix an $m$. Then, there exists an $n_0$ such that for all $n\geq n_0$
  there exists a sequence
 $\left\{ n_{\mathbf{i}} \right\}_{\mathbf{i}\in\mathcal{T}_m }$ such that
 \begin{equation}
 \label{x66}
 1\leq n_{\mathbf{i}}< n,\quad
 \sum_{\mathbf{i}\in \mathcal{T}_m}n_{\mathbf{i}}
 \leq 2n,\quad
 u_{n,r}(\nu)\geq
 \sum_{\mathbf{i}\in \mathcal{T}_m}
 \psi(\mathbf{i}) 3^{-|\mathbf{i}|r}
 u_{n_{\mathbf{i}},r}(\nu).
 \end{equation}
\end{lemma}

\begin{proof}
   Fix an $m$ and to shorten the notation we write
  $\Gamma :=\mathcal{T}_m$.
  We consider any $\pmb{\tau}\in\Sigma ^*$ such that
$S_{\pmb{\tau}}(I)\subset U$. For example $\pmb{\tau}=(1)$ will do since $I_1=S_1(I)=\left[1,\frac{5}{2}\right]\subset \left( 0,\frac{9}{2} \right)$. Let
$\varepsilon :=\text{dist}(S_{\tau }(I),U^c)$. If we choose $\tau =1$, then $\varepsilon =1$. Let $\delta :=3^{-m}$. Then,
\begin{equation}
\label{x63}
d(x,U^c)\geq d(x, S_{\mathbf{i}}(U^c))\geq \delta \varepsilon ,\quad
\text{ if }
x\in I_{\mathbf{i}\pmb{\tau}} \text{ and }
\mathbf{i}\in \mathcal{T}_m.
\end{equation}
For each $n$ let $\alpha _n$ be the optimal set for
$u_{n,r}(\nu)$. This exists according to Lemma \ref{x65}. We define
\begin{equation}
\label{x62}
\delta _n:=
\max\left\{
  d(x,\alpha _n\cup U^c):x\in \Lambda
 \right\}.
\end{equation}
Since $\delta _n\to 0$ we can choose an $n_0$ such that
$\delta _n<\delta  \varepsilon $ if $n\geq n_0$.
Let $x\in \Lambda _{\mathbf{i}\pmb{\tau}}$.
Then, $x\in S_{\mathbf{i}}(U)\subset U$, for all
$\mathbf{i}\in \Gamma $. By compactness,
there exists an $a\in\alpha _n\cup U^c$ such that
$$
d(x,\alpha _n\cup U^c)=d(x,a)\leq \delta _n<\delta \varepsilon ,
$$
where the one but last inequality holds since
$x\in\Lambda $. So, by \eqref{x63} we have
$a\not\in S_{\mathbf{i}}(U^{c})$. That is, $a\in S_{\mathbf{i}}(U)\subset U$ but $a\in \alpha _n\cup U^c$,
hence $a\in \alpha _n$. Let
$$
n_{\mathbf{i}}:=\#
\alpha _n\cap S_{\mathbf{i}}(U),\quad
\alpha _{n_{\mathbf{i}}}:= \alpha _n\cap S_{\mathbf{i}}(U)
.$$
Then,  we have just proved that
$a\in \alpha _{\mathbf{n}_\mathbf{i}}$. That is,
\begin{equation}
\label{x61}
n_{\mathbf{i}}\geq 1, \quad \forall \mathbf{i}\in \Gamma .
\end{equation}
It follows from the optimal property of $\alpha _n$ that $\alpha _n\not\subset \alpha_{n_{\mathbf{i}}}$.
Hence, $n_{\mathbf{i}}<n$ for all $\mathbf{i}\in\Gamma $.

It follows from Property-\ref{w96} that
\begin{equation}
\label{x60}
\#\left\{ \mathbf{i}\in \Gamma :
x\in S_{n_{\mathbf{i}}}(I)
\right\}\leq 2,\quad \forall   x\in I.
\end{equation}
Thus,
\begin{equation}
\label{x56}
\sum_{\mathbf{ i}\in \Gamma }n_{\mathbf{i}}\leq
2n.
\end{equation}
Finally we prove that
\begin{equation}
\label{x50}
u_{n,r}(\nu)\geq
\sum_{\mathbf{ i}\in\Gamma }\psi(\mathbf{i})
  3^{-mr} u_{n_{\mathbf{i}},r}(\nu).
\end{equation}
To verify this
  we will use the following trivial observation
\begin{equation}
\label{x55}
\alpha _n\cup S_{\mathbf{i}}(U^c)=
\alpha _n\cup (S_{\mathbf{i}}(U))^c
=
\left( \alpha _n\cap S_{\mathbf{i}}(U) \right)
\cup
\left(S_{\mathbf{i}}(U)  \right)^c.
\end{equation}
Using this we get
  \begin{multline}
    \nonumber  \int\! d\left(S_{\mathbf{i}}(x),
    \alpha _n\cup S_{\mathbf{i}}(U^c)
    \right)^rd\nu (x) \!\!=
    \!\!
    \int\! d\left(S_{\mathbf{i}}(x),
    (\alpha _n\cap S_{\mathbf{i}}(U))\cup S_{\mathbf{i}}(U)^c
    \right)^rd\nu (x)\\
  =
  3^{-mr} \int d\left(x,
S _{\mathbf{ i}}^{-1}\bigg[
  \underbrace{
  \big( \alpha _n\cap S_{\mathbf{i}}(U)\big)
  }_{\alpha _{n_{\mathbf{i}}}}
 \cup S_{\mathbf{i}}(U^c)\bigg]
    \right)^rd\nu (x)\\
   = 3^{-mr} \int d\left( x,
    S _{\mathbf{ i}}^{-1}(\alpha _{n_{\mathbf{i}}})
    \cup U^c
    \right)^r d\nu (x)\geq
    3^{-mr} u_{n_{\mathbf{i}},r}(\nu).
  \end{multline}
Now we put all of these together:
\begin{eqnarray}\label{x52}
\nonumber u_{n,r}(\nu)    &=&
 \int d(x,\alpha _n\cup U^c)^rd\nu (x)
 =
 \sum_{\mathbf{ i}\in\Gamma }\psi(\mathbf{i})
 \int d\left(
S_{\mathbf{i}}(x),\alpha _n\cup U^c
  \right)^rd\nu (x)
 \\
     &\geq&
     \sum_{\mathbf{ i}\in\Gamma }\psi(\mathbf{i})
     \int d\left(
    S_{\mathbf{i}}(x),\alpha _n\cup
    S_{\mathbf{i}}( U^c)
      \right)^rd\nu (x)
  \geq
  \sum_{\mathbf{ i}\in\Gamma }\psi(\mathbf{i})
  3^{-mr} u_{n_{\mathbf{i}},r}(\nu),
  \end{eqnarray}
where at the second step we used \eqref{x53},
at the third step we used that $S_{\mathbf{i}}(U^c)\supset U^c$. At the fourth step we used
the previous inequality.
\end{proof}
\begin{proposition}\label{x49}
  Let $0<\ell <\chi_r$. Then,
  \begin{equation}
  \label{x48}
  \liminf\limits_{n\to\infty}
  n\cdot \left( u_{n,r}(\nu) \right)^{\frac{\ell }{r}}
  >0.
  \end{equation}
\end{proposition}
Our proof follows the line
  \cite[Lemma 4.4]{GL3} and
\cite[Proposition 3.12]{R3}.

\begin{proof}
  Fix an $\ell \in\left( 0,\chi_r \right)$. Then,
  $\frac{\ell }{r+\ell }<t_r$. Recall the definition of the pressure function $P(t)$
  from \eqref{u13}. It follows from \eqref{y35}
  that
\begin{equation}
\label{x47}
P\left(\frac{\ell }{r+\ell }  \right)
=
\lim\limits_{m\to\infty}\frac{1}{m}
\log \sum_{\mathbf{ i}\in\mathcal{T}_m}
\left( \psi(\mathbf{i})3^{-mr} \right)^{\frac{\ell }{r+\ell }}>0.
\end{equation}
Fix an $m$ such that
\begin{equation}
\label{y40}
\sum_{\mathbf{ i}\in\mathcal{T}_m}
\left(
\psi(\mathbf{i}) 3^{-mr}
 \right)^{\frac{\ell }{r+\ell }}>2.
\end{equation}
For this $m$ we choose $n_0$ as in Lemma \ref{x67}.
Let
\begin{equation}
\label{x39}
C:=\min\left\{
  q^{\frac{r}{\ell }}u_{q,r}(\nu):q\leq n_0
 \right\}.
\end{equation}
Clearly, $u_{n,r}(\nu)>0$. Hence, $C>0$.
Choose an  $n$ such that
\begin{equation}
\label{x38}
n\geq n_0\quad \&\ \quad
 k< n\Longrightarrow
k^{\frac{r}{\ell }}u_{k,r}(\nu)\geq C.
\end{equation}
Below we prove that
\begin{equation}
\label{x37}
n^{\frac{r}{\ell }}u_{n,r}(\nu)\geq C.
\end{equation}
For this $m$ and $n$ we choose $n_{\mathbf{i}}$ for every $\mathbf{i}\in \mathcal{T}_m$, as in Lemma \ref{x67},
such that the inequalities of \eqref{x66} hold.
\begin{eqnarray}\label{x36}
  \nonumber n^{\frac{r}{\ell }}u_{n,r}(\nu)   &\geq&
  n^{\frac{r}{\ell }}
  \sum_{\mathbf{ i}\in \Gamma }
  \psi(\mathbf{i})3^{-mr} u_{n_{\mathbf{i}},r}(\nu)
  \\
     &=&  n^{\frac{r}{\ell }}
  \sum_{\mathbf{ i}\in \Gamma }
  \psi(\mathbf{i})3^{-mr} u_{n_{\mathbf{i}},r}(\nu)
  (n_{\mathbf{i}})^{-\frac{r}{\ell }}
 \underbrace{ (n_{\mathbf{i}})^{\frac{r}{\ell }}
  u_{n_{\mathbf{i}},r}(\nu)}_{\geq C}
  \\
  &\geq& C
 \underbrace{ \sum_{\mathbf{ i}\in \Gamma }
  \psi(\mathbf{i})3^{-mr}\left(
\frac{n_{\mathbf{i}}}{n}
   \right)^{-\frac{r}{\ell }}}_{a_n},
  \end{eqnarray}
  where at first step we used \eqref{x52}, and at the last step we used that $n_{\mathbf{i}}<n$ so by
  \eqref{x38} we have
  $u_{n_{\mathbf{i}},r}(\nu)\geq C$.
Set $a_n:=\sum\limits_{\mathbf{ i}\in \Gamma }
\psi(\mathbf{i})3^{-mr}\left(
\frac{n_{\mathbf{i}}}{n}
 \right)^{-\frac{r}{\ell }}$. To verify \eqref{x37},
 it is enough to prove that
\begin{equation}
\label{x35}
a_n\geq 1.
\end{equation}
To see this we use the so-called  reversed H\"older inequality: Let $\left\{ x_k \right\}_{k=1}^{M }$ and
$\left\{ y_k \right\}_{k=1}^{M}$ be finite sequences of positive numbers and let $p\in (1,\infty  )$. Then,
\begin{equation}
\label{x34}
\sum_{k=1}^{M} x_k y_k\geq
\left(  \sum_{k=1}^{M}
x_{k}^{\frac{1}{p} }
\right)^p
\left( \sum_{k=1}^{M}
y_{k}^{-\frac{1}{p-1} }
\right)^{-(p-1)}.
\end{equation}
In our case the summation is taken for $\mathbf{i}\in \Gamma $, $x_{\mathbf{i}}=\psi(\mathbf{i})3^{-mr}$
 and $y_{\mathbf{i}}=\left(
  \frac{n_{\mathbf{i}}}{n}
   \right)^{-\frac{r}{\ell }}$. Finally, $p:=1+\frac{r}{\ell }$. That is, $\frac{1}{p}=\frac{\ell }{r+\ell }$, $\frac{1}{p-1}=\frac{\ell }{r}$ and $-(p-1)=-\frac{r}{\ell }$. Hence, from the reversed H\"older inequality we get
   \begin{equation}
   \label{x33}
 a_n\geq
 \left(\sum_{\mathbf{ i}\in \Gamma }
 \left( \psi(\mathbf{i})3^{-mr} \right)^{\frac{\ell }{r+\ell }}
 \right) ^{1+\frac{r}{\ell }}
 \left(\sum_{\mathbf{ i}\in \Gamma }
 \left(\frac{n_{\mathbf{i}}}{n}  \right)^{\left(-\frac{r}{\ell }  \right)\left(-\frac{\ell }{r}  \right)}
 \right)^{-\frac{r}{\ell }} .
   \end{equation}
Putting together this, \eqref{x56}  and \eqref{y40}
we get that
$$
a_n\geq 2^{1+\frac{r}{\ell }}\cdot
2^{-\frac{r}{\ell }}  =2.
$$
In this way we have proved that \eqref{x35} holds which implies that $n^{\frac{r}{\ell }}u_{n,r}(\nu) >C$
for all $n$. Thus, \eqref{x48} holds.
\end{proof}
Clearly, $u_{n,r}(\nu)\leq V_{n,r}(\nu)$. Hence, we get from Proposition \ref{x49} that
\begin{equation}
\label{x32}
\liminf\limits_{n\to\infty} n\cdot e _{n,r}^{\ell  }
>0.
\end{equation}
Combining this with \eqref{w86} we obtain that
\begin{equation}
\label{x31}
\chi _r\leq\underline{D}_r(\nu ).
\end{equation}
\begin{proof}[Proof of Theorem \ref{x70}]
Putting together \eqref{x30} and \eqref{x31}
we obtain that
\begin{equation}
\label{x29}
D_r(\nu )=\chi _r.
\end{equation}
\end{proof}

\section{Checking Properties 1-4}\label{aa99}


\subsection{Good blocks and bad blocks}
Let $n\in \mathbb{N} \cup \{\infty \} $.
For  arbitrary $1\leq k<\ell\leq n$ and we use the shorthand notation
$$
[n]:=\left\{ 1,\dots  ,n \right\},\quad
[k,\ell ]:=\left\{ k,\dots  ,\ell  \right\}.
$$

\begin{definition}
  \textbf{Bad blocks, Good blocks} \\
  Let $n\in \mathbb{N} \cup \{\infty\}  $ and $\pmb{\eta}=(\eta _1,\dots  ,\eta _n)$.
  For a $k<\ell \leq n$
we say that $[k,\ell ]$ is a \texttt{bad block}  of $[n]$
with respect to an $\pmb{\eta }=(\eta _1,\dots  ,\eta _n)\in \mathcal{A}^n$
if
$(\eta _k,\dots  ,\eta _{\ell })=(0,3,3,\dots  ,3)$. We say that $[k,\ell ]$ is a \texttt{maximal bad block} if either $\ell =\infty  $ or $\eta _{\ell +1}\ne 3$. Similarly, $[k,\ell ]$
is a \texttt{good block} of $[n]$ with respect to $\pmb{\eta }$
if $(\eta _k,\dots  ,\eta _{\ell })=(1,1,1,\dots  ,1,0)$.
We say that $[k,\ell ]$ is a \texttt{maximal good block} if either $k=1  $ or $\eta _{k-1}\ne 1$.
We write $B(\pmb{\eta })$, ($G(\pmb{\eta })$) for the collection of maximal  bad (good) blocks with respect to $\pmb{\eta }$, respectively. That is,
\begin{multline}
\label{w85}
B(\pmb{\eta }):=
\left\{
[k,\ell ]:
1\leq k< \ell <\infty  ,\
\eta _k=0,\
\eta _{k+1}=\cdots=\eta _{\ell }=3,
\eta _{\ell +1}\ne 3
\right\}
\\
\bigcup
\left\{ [k,\infty  ]: 1\leq k,\
\eta _k=0,\
\eta _{k+1}=\eta _{k+2}=\eta _{k+3}=\cdots =3
\right\}.
\end{multline}
Similarly,
\begin{multline}
  \label{w84}
  G(\pmb{\eta }):=
  \left\{
  [k,\ell ]:
  1\leq k< \ell <\infty  ,\
  \eta _k=\cdots=\eta _{\ell -1}=1,\
  \eta _{\ell }=0,
  \eta _{k-1}\ne 1
  \right\}
  \\
  \bigcup
  \left\{ [k,\infty  ]: 1\leq k,\
  \eta _k=
  \eta _{k+1}=\eta _{k+2}=\eta _{k+3}=\cdots =1
  \right\}.
  \end{multline}
\end{definition}

\begin{definition}
  \textbf{A "bad" and a "good" partition of $[n]$} \\
  Given an $n\in \mathbb{N} \cup \{\infty \}$ and a
  $\pmb{\eta }\in \mathcal{A}^n$, the following definitions are meant to be with respect to $\pmb{\eta }$.
  \begin{enumerate}
    [{\bf (1)}]
      \item
    $
    A_{\text{Good}}(\pmb{\eta }):=[n]\setminus \bigcup\limits _{[k,\ell ]\in G(\pmb{\eta }) }[k,\ell ], \text{ and }
    A_{\text{Bad}}(\pmb{\eta }):=[n]\setminus \bigcup\limits_{[k,\ell ]\in B(\pmb{\eta }) }[k,\ell ],
    $
    \item $
    B_{\text{Good}}(\pmb{\eta }):=\bigcup\limits _{[k,\ell ]\in G(\pmb{\eta })}
    \left\{ \ell  \right\}, \text{ and }
    B_{\text{Bad}}(\pmb{\eta }):=\bigcup\limits _{[k,\ell ]\in B(\pmb{\eta })}
    \left\{ \ell  \right\},
    $
    \item
   $C_{\text{Good}}(\pmb{\eta }):=\bigcup\limits _{[k,\ell ]\in G(\pmb{\eta }) }[k,\ell -1],
    \text{ and }
    C_{\text{Bad}}(\pmb{\eta }):=\bigcup\limits _{[k,\ell ]\in B(\pmb{\eta }) }[k,\ell -1]
    $,
\item $D_{\text{Good}}(\pmb{\eta }):=
A_{\text{Good}}(\pmb{\eta })\cup B_{\text{Good}}(\pmb{\eta })
$, \text{ and }
$D_{\text{Bad}}(\pmb{\eta }):=
A_{\text{Bad}}(\pmb{\eta })\cup B_{\text{Bad}}(\pmb{\eta })
$.
    \end{enumerate}
    We use most frequently $C_{\text{Bad}}(\pmb{\eta })$ and
    $D_{\text{Bad}}(\pmb{\eta })$, so we also explain their meaning in words:

    $C_{\text{Good}}(\pmb{\eta })$ is the collection of
    indices which are in a good box of $[n]$ with respect to $\pmb{\eta }$ but not as a right endpoint of a maximal good box (with respect to $\pmb{\eta }$).

    $D_{\text{Good}}(\pmb{\eta })$ is the collection of
    indices which are either the right endpoint of a  maximal good box of $[n]$ or
    not contained in any good boxes of $[n]$ (with respect to $\pmb{\eta }$).
  \end{definition}
  In this way we partition $[n]$ into
  $C_{\text{Good}}(\pmb{\eta })\cup D_{\text{Good}}(\pmb{\eta })$. The indices in $D_{\text{Good}}(\pmb{\eta })$
cause less problem than the ones in $C_{\text{Good}}(\pmb{\eta })$. This is indicated by the following fact. Before stating it
recall that $\mathcal{I}_{\mathbf{i}}$  was defined in \eqref{u99}. Using this definition we get that
for an $\mathbf{i}\in\Sigma _A^*$
\begin{equation}
\label{y03}
\mathcal{I}_{\mathbf{i}}=
\left\{ \pmb{\eta}\in \mathcal{A}^{|\mathbf{i}|}:
\Pi (\pmb{\eta})=\Pi (\mathbf{i})
\right\}.
\end{equation}
\begin{fact}\label{y76}
  Let $\mathbf{i}\in\mathcal{T}_n$ and let $z\in D_{\text{Good}}(\mathbf{i})$. Then, for every $\pmb{\eta}\in \mathcal{I}_{\mathbf{i}}$ we have
  \begin{equation}
  \label{y75}
  \pmb{\eta}|_z\in \mathcal{I}_{\mathbf{i}|_z},
  \quad
  \sigma ^z \pmb{\eta}\in \mathcal{I}_{\sigma ^z \mathbf{i}}
  \end{equation}
  where $\pmb{\tau}|_z:=(\tau _1,\dots  ,\tau _z)$ if
$\pmb{\tau}=(\tau _1,\dots    ,\tau _n)\in\mathcal{A}_n$ and
$z\leq n$.
\end{fact}

 Note that
 for all $\mathbf{a},\mathbf{b},\mathbf{c},\mathbf{d}\in\Sigma _A^*$
\begin{equation}
\label{y79}
\Pi (\mathbf{a})=\Pi (\mathbf{b})\  \&\
\Pi (\mathbf{c})=\Pi (\mathbf{d})
\Longrightarrow
\Pi (\mathbf{a},0,\underbrace{3,\dots  ,3}_{\ell -1},\mathbf{c})
=
\Pi (\mathbf{b},\underbrace{1,\dots  ,1}_{\ell -1},0,\mathbf{d}).
\end{equation}
This is immediate from  the definition \eqref{y96} of the natural projection $\Pi $ since
\begin{equation}
\label{y30}
\Pi (0,\underbrace{3,\dots  ,3}_{\ell -1})
=
0\cdot 3^{-1}+\sum_{k=2}^{\ell  }3\cdot 3^{-(k-1)}
=
\sum_{k=1}^{\ell-1  }1\cdot 3^{-(k -1)}+0\cdot 3^{-(\ell -1)}
=
\Pi (\underbrace{1,\dots  ,1}_{\ell -1},0).
\end{equation}
Now we summarize some important properties of the mapping $\Phi $ introduced in Definition \ref{w91}:
\begin{fact}\label{y78}
 \begin{enumerate}
 [{\bf (a)}]
   \item For every $\pmb{\eta}\in \mathcal{A}^n$ we have $\Phi (\pmb{\eta})\in \mathcal{T}_n$,
   \item $\Pi (\pmb{\eta})=\Pi (\Phi (\pmb{\eta}))$
   if $\pmb{\eta}\in \mathcal{A}^n$,
   \item $\mathcal{I}_{\mathbf{i}}=\left\{ \pmb{\eta}\in\mathcal{A}^n:\Phi (\pmb{\eta})=\mathbf{i} \right\}$,
    for any  $\mathbf{i}\in \mathcal{T}_n$.
 \end{enumerate}
\end{fact}
\begin{proof}
  Part (a) is obvious from Definition \ref{w91} since we kill all bad blocks of $\pmb{\eta}$.

  To prove part (b) we apply \eqref{y79}
  with $\mathbf{a}=\mathbf{c}$ and $\mathbf{b}=\mathbf{d}$ in every step of the construction
  of $\Phi (\pmb{\eta})$ in Definition \ref{w91}.

  To prove part (c), observe that the inclusion
  "$\supset $" follows from part (b) and \eqref{y03}. In order to verify the  inclusion "$\subset $" in (c),
  we assume the opposite to get a contradiction.
  That is, we assume that there exists an $\pmb{\eta}\in \mathcal{A}^{|\mathbf{i}|}$ such that
  $\pmb{\eta}\in \mathcal{I}_{\mathbf{i}}$
  but $\Phi (\pmb{\eta})\ne \mathbf{i}$.
  But as we saw in part (b), $\Pi (\pmb{\eta})=
  \Pi (\Phi (\pmb{\eta}))$. In this way the distinct
  $\mathbf{i},\Phi (\pmb{\eta})\in \mathcal{T}_n$ satisfy $\Pi (\Phi (\pmb{\eta}))= \Pi (\mathbf{i})$. This is impossible by part (a) of Fact \ref{y90}.
\end{proof}
Part (c) of Fact \ref{y74} and the definition of $\Phi $
(Definition \ref{w91}) imply that
\begin{equation}
\label{y25}
\mathbf{i}\in \mathcal{T}_n\quad \& \quad
G(\mathbf{i})=\emptyset \quad
\Longrightarrow\quad
\mathcal{I}_{\mathbf{i}}
=\left\{ \mathbf{i} \right\}.
\end{equation}
Similarly, the following inequality is an immediate consequence of the definition of $\Phi $ and part (c) of Fact \ref{y78}
\begin{equation}
\label{c97}
\#\mathcal{I}_n\leq n,\quad \text{for every } n.
\end{equation}

We partition $\mathcal{A}^n$ as follows
\begin{equation}
\label{y77}
\mathcal{A}^n=\bigcup _{\mathbf{i}\in\mathcal{T}_n}
\mathcal{I}_{\mathbf{i}}, \text{ where }
\mathcal{I}_{\mathbf{i}}:=
\left\{ \pmb{\eta}\in\mathcal{A}^n: \Pi (\pmb{\eta})=\Pi (\mathbf{i}) \right\}=
\left\{ \pmb{\eta}\in\mathcal{A}^n:
\Phi (\pmb{\eta})=\mathbf{i} \right\}.
\end{equation}
Using Fact \ref{y78} and \eqref{y86} this is a partition indeed.

\begin{proof}[Proof of Fact \ref{y76}]
  Let $\pmb{\eta}\in\mathcal{A}^n$ for an $n\geq 2$.
  By part (c) of Fact \ref{y78} we have
$\Phi (\pmb{\eta})=\mathbf{i}$.
Then, by the definition of the mapping $\Phi $
we get $\mathbf{i}$ from $\pmb{\eta}$
by replacing all maximal bad blocks of $\pmb{\eta}$ by the corresponding good blocks. That is, if  $[k,\ell ]\in B(\pmb{\eta})$, then we
define  $i_k=\cdots =i_{\ell -1}=1$ and $i_\ell =0$.
Every bad block in $\pmb{\eta}$ is a good block in $\mathbf{i}$.
If we stop at an index $z\in D_{\text{Good}}(\mathbf{i})$, then the collection of maximal good blocks of
$\mathbf{i}|_z$ are the same as the collection of
those maximal good blocks of $\mathbf{i}$ which intersect $[z]=\left\{ 1,\dots  ,z \right\}$.
So, if we apply the definition of $\Phi $ to
$\pmb{\eta}|_z$ we get that
$\Phi (\pmb{\eta}|_z)=\mathbf{i}|_z$.
The second part follows from the first part and the definition of $\Pi $.
\end{proof}

Using a little modification of this argument we can also prove that

\begin{fact}\label{y13}
  Let $\pmb{\eta}\in \mathcal{A}^n$, where $n\in \mathbb{N}\cup \{\infty\}$. Let $z<n$ and $\mathbf{i}:=\Phi (\pmb{\eta})|_z$.
\begin{enumerate}
[{\bf (a)}]
  \item  If $z\in D_{Bad}(\pmb{\eta})$, then
  \begin{equation}
  \label{y12}
  \Phi (\pmb{\eta}|_z)=\Phi (\pmb{\eta})|_z \quad \text{ that is }\quad
  \pmb{\eta}|_z\in \mathcal{I}_{\mathbf{i}}.
  \end{equation}
  \item If  $z\in C_{Bad}(\pmb{\eta})$, then the z-th coordinate of
$\Phi (\pmb{\eta})$ is equal to $1$ and
  \begin{equation}
  \label{a97}
  \Phi (\pmb{\eta}|_z)=\Phi (\pmb{\eta})|_{z-1}0\ne \Phi (\pmb{\eta})|_{z} \quad
\text{  that is  }
  \quad
  \pmb{\eta}|_z\in \mathcal{I}_{\mathbf{i}^-0}.
  \end{equation}
\item Consequently, we get that for all $z< n$ we have
\begin{equation}
\label{a96}
\pmb{\eta}|_z\in \mathcal{I}_{\mathbf{i}}\cup
 \mathcal{I}_{\mathbf{i}^-0}.
\end{equation}
\end{enumerate}

\end{fact}

\begin{example}\label{y74}
  Let
  $\pmb{\eta}=(0,\underbrace{3,\dots  ,3}_{n-1})$,
  $\mathbf{i}=(\underbrace{1,\dots  ,1}_{n-1},0)$ and let
  $k=n-1$. Moreover, let $\pmb{\tau}^\ell :=(\underbrace{1,\dots  1}_{\ell },0,\underbrace{3,\dots  ,3}_{n-1-\ell })$, where $0\leq\ell \leq n-1$.
  Then,
  $$\mathcal{I}_{\mathbf{i}}=\left\{ \pmb{\tau}^0,\dots  ,\pmb{\tau}^{n-1}  \right\}
  \quad
  \text{ and }\quad
  \mathcal{I}_{\mathbf{i}^-}=\big\{
    (\underbrace{1,\dots  ,1}_{n-1})
   \big\},
  $$
  where $\mathbf{i}^-$ was defined in \eqref{y73}.
\end{example}
\noindent
That is, there exists $n-1$ elements  $\pmb{\eta}\in \mathcal{I}_{\mathbf{i}}$  such that
$\Phi (\pmb{\eta}|_{n-1})\not\in \mathcal{I}_{\mathbf{i}|_{n-1}}$.

\begin{fact}\label{y70}
  Let $\mathbf{i}\in\mathcal{T}_n$, $\pmb{\eta}\in \mathcal{I}_{\mathbf{i}}$.
  Let $[k,\ell ]\in G(\mathbf{\mathbf{i}})$
  and $[k',\ell']\in B(\pmb{\eta})$ such that
  $[k,\ell ]\cap [k',\ell']\ne \emptyset $. Then,
  $\ell =\ell '$.
\end{fact}
Namely, $\ell '>\ell $ is not possible since this would follow that $i_\ell =0$, $i_{\ell +1}=3$. $\ell '<\ell $ is not possible since this would follow $k\leq \ell '<\ell $
and
$i_{\ell '}=0$  which is not possible since by definition $i_{\ell '}=1$.

\subsection{The cases when $\psi(\mathbf{i})$ is multiplicative}\label{y59}
For the rest of the paper we fix some notation.
Let
\begin{equation}
\label{y26}
\chi:=\chi(\mathbf{i}):=
\left\{
\begin{array}{ll}
\max\left\{ \ell :
[k,\ell ]\in G(\mathbf{i}) \right\}
,&
\hbox{if $G(\mathbf{i})\ne \emptyset $;}
\\
0
,&
\hbox{otherwise.}
\end{array}
\right.
\end{equation}
Moreover, let
\begin{equation}
  \label{u89}
  \xi:=\xi(\mathbf{j}):=
  \min\left\{
  k: \forall \ell >k,\ j_{\ell }=1
   \right\}.
  \end{equation}

Clearly,
\begin{equation}
\label{y42}
G(\mathbf{i})\ne \emptyset \Longrightarrow
(i_{\chi-1},i_{\chi})=(1,0).
\end{equation}

We also get by definition  that
\begin{equation}
\label{u60}
\chi (\pmb{\ell })\leq \xi (\pmb{\ell }), \quad
\forall \pmb{\ell }\in\mathcal{T}_p, \quad \forall p.
\end{equation}

\begin{fact}\label{y27}
For $\ell \geq \chi$
\begin{equation}
\label{y24}
\psi (\mathbf{i})=\psi (\mathbf{i}|_{\ell })\cdot p_{i_{\ell +1}\dots  i_n}.
\end{equation}
\end{fact}
\begin{proof}
 If $G(\mathbf{i})\ne\emptyset$, then
the proof is immediate from Definition \ref{w91}.  In the case when $G(\mathbf{i})=\emptyset$,
then $\chi=0$, $\psi (\mathbf{i}|_0)=\psi (\flat)=1$.
Then,
the fact follows from \eqref{y25}.
\end{proof}

\begin{lemma}\label{u77}
 Let $n\geq 2$, $\mathbf{i}\in\mathcal{T}_n$ and
  $z\in D_{\text{Good}}(\mathbf{i})
 $. Then,
 \begin{equation}
 \label{u76}
 \psi(\mathbf{i})=
 \psi(\mathbf{i}|_z)\cdot
 \psi(\sigma ^z\mathbf{i}).
 \end{equation}
\end{lemma}
\begin{proof}
  $$
\psi(\mathbf{i})=
\sum_{\pmb{\eta} \in \mathcal{I}_{\mathbf{i}}}
p_{\pmb{\eta}}
=
\sum_{\pmb{\eta} \in \mathcal{I}_{\mathbf{i}}}
p_{\pmb{\eta}|_z} \cdot
p_{\sigma ^z \pmb{\eta}}
\leq
\sum_{\pmb{\omega} \in \mathcal{I}_{\mathbf{i}|_z}}
p_{\pmb{\omega}}
\cdot
\sum_{\pmb{\tau} \in \mathcal{I}_{\sigma ^z\mathbf{i}}}
p_{\pmb{\tau}}
=
\psi(\mathbf{i}|_z)\cdot
\psi(\sigma ^z\mathbf{i}),
  $$
where in the third step we used that
by Fact \ref{y76}
for an $\pmb{\eta}\in \mathcal{I}_{\mathbf{i}}$
we have
$\pmb{\eta}|_z\in \mathcal{I}_{\mathbf{i}|_z}$
and
$\sigma ^z\pmb{\eta}\in \mathcal{I}_{\sigma ^z\mathbf{i}}$ since $z\in D_{\text{Good}}(\mathbf{i})
$.
  On the other hand,
$$
\psi(\mathbf{i}|_z)\cdot
\psi(\sigma ^z\mathbf{i})
=
\sum_{\pmb{\omega} \in \mathcal{I}_{\mathbf{i}|_z}}
p_{\pmb{\omega}}
\cdot
\sum_{\pmb{\tau} \in \mathcal{I}_{\sigma ^z\mathbf{i}}}
p_{\pmb{\tau}}
=
\sum_{\pmb{\omega} \in \mathcal{I}_{\mathbf{i}|_z}
\atop
\pmb{\tau} \in \mathcal{I}_{\sigma ^z\mathbf{i}}
}p_{\pmb{\omega}\pmb{\tau}}
\leq
\sum_{\pmb{\eta} \in \mathcal{I}_{\mathbf{i}}}
p_{\pmb{\eta}}
=
\psi(\mathbf{i}),
$$
where in the one but last step we used that
$(i_{z},i_{z+1})\ne (0,3)$ since $\mathbf{i}\in \mathcal{T}_n$.
Moreover, we also used that
by \eqref{y96},
\begin{equation}
\label{w83}
\Pi (\pmb{\eta })=\Pi (\pmb{\eta }|_z)+3^{-z}\Pi (\sigma ^z\pmb{\eta }),\quad \forall \pmb{\eta }\in\mathcal{A}^n.
\end{equation}
Hence, $\pmb{\omega }\in \mathcal{I}_{\mathbf{i}|_z}$
and $\pmb{\tau}\in \mathcal{I}_{\sigma ^z\mathbf{i}}$
implies that $\pmb{\eta}=\pmb{\omega}\pmb{\tau}\in \mathcal{I}_{\mathbf{i}}$.
\end{proof}

\subsection{The properties of the sequence $\left\{a_w \right\}_w$}\label{w80}
Observe that for a $w\geq  1$
\begin{equation}
\label{u84}
\mathcal{I}_{\overline{1}^w}=\left\{ \overline{1}^w \right\}
\quad
\text{ and }
 \quad
\mathcal{I}_{\overline{1}^{w-1}0}
=
\left\{
\overline{1}^{w-1-\ell }0
\overline{3}^{\ell }
 \right\}_{\ell =0}^{w-1 }.
\end{equation}
Hence,
\begin{equation}
  \label{u43}
  \psi (\overline{1}^{w})=
  p _{1}^{w}
  \quad
  \text{ and }
  \quad
  \psi \left( \overline{1}^{w-1}0 \right)=
  a_w:
  =
  \frac{p_0}{p_1}p _{1}^{w}
  \sum_{\ell =0}^{w-1 }
  \left( \frac{p_3}{p_1} \right)^{\ell }.
  \end{equation}
So,
  \begin{equation}
    \label{u85}
\psi \left( \overline{1}^{w-1}0 \right)
    >
    \psi (\overline{1}^{w})
\Longleftrightarrow
    \sum\limits_{\ell =0}^{w-1 }\left( \frac{p_3}{p_1} \right)^{\ell }>\frac{p_1}{p_0}.
    \end{equation}
This motivates the following definition.
For an $n\geq 1$ and $\mathbf{i}\in \mathcal{T}_n$ we define
\begin{equation}
  \label{u87}
  \mathbf{i}^*:=
  \left\{
  \begin{array}{ll}
  \mathbf{i}^-0
  ,&
  \hbox{if\qquad  $n-\xi(\mathbf{i})\geq 1$ and  $\sum\limits_{\ell =0}^{n-\xi(\mathbf{i})-1 }\left( \frac{p_3}{p_1} \right)^{\ell }>\frac{p_1}{p_0}$;}
  \\
  \mathbf{i}
  ,&
  \hbox{otherwise.}
  \end{array}
  \right.
  \end{equation}
We introduce
\begin{equation}
\label{u59}
\mathfrak{A}:=
\left\{
  q\in \mathbb{N}:
  q\geq 1\  \& \
  \sum\limits_{\ell =0}^{q-1 }\left( \frac{p_3}{p_1} \right)^{\ell }>\frac{p_1}{p_0}
 \right\}.
\end{equation}
Observe that
\begin{equation}
\label{u58}\mathbf{i}^*\ne \mathbf{i}
 \Longleftrightarrow
|\mathbf{i}|-\xi (\mathbf{i}) \in \mathfrak{A}.
\end{equation}
Recall that in \eqref{y87} we assumed that $p_3\leq p_1$.
A simple calculation shows that
\begin{equation}
\label{u57}
\mathfrak{A}\ne \emptyset
\Longleftrightarrow
p_1-p_3<p_0.
\end{equation}
Putting together \eqref{u85}
and \eqref{u59} we get that
  \begin{equation}
  \label{u63}
 q\in \mathfrak{A}
  \Longleftrightarrow
  \widehat{\psi}(\overline{1}^q)> \psi(\overline{1}^q).
  \end{equation}\
  We define
\begin{equation}
\label{u81}
q_0:=
\left\{
\begin{array}{ll}
\min \mathfrak{A}
,&
\hbox{if $p_1-p_3
<p_0$;}
\\
\infty
,&
\hbox{otherwise.}
\end{array}
\right.
\end{equation}

Now we prove the sub-multiplicative property of the sequence $\left\{ a_w \right\} _{w=1}^{\infty}$.

\begin{fact}\label{u49}
   Let $u,v\geq 1$. Then, there are constants $C_6, C_7>0$ such that
\begin{equation}
\label{u48}
a_{u+v}<C_6 a_{u}a_v,
\end{equation}
and
\begin{equation}
\label{u47}
\sum_{\ell =0}^{u+v }
\left( \frac{p_3}{p_1} \right)^{\ell }
<
C_7\cdot
\sum_{\ell =0}^{u-1}
\left( \frac{p_3}{p_1} \right)^{\ell }
\cdot
\sum_{\ell =0}^{v -1}
\left( \frac{p_3}{p_1} \right)^{\ell }.
\end{equation}
\end{fact}

\begin{proof}
  First we prove \eqref{u47}. Recall
  that we assumed that $0<p_3\leq p_1$. If $p_1=p_3$, then we can clearly choose
  a constant $C_7 >0$ such that $u+v+1< C_7\cdot u\cdot v$.
  If $p_3<p_1$, then \eqref{u47} holds with the choice of $C_7=\frac{1}{1-\frac{p_3}{p_1}}$. It is immediate that \eqref{u48}
  follows from \eqref{u47}.
\end{proof}

The following fact is an important but trivial consequence of the definitions.
\begin{fact}\label{u42}
Given a  $\mathbf{k}\in \mathcal{T}_n$  with $k_n\ne 1$ and a natural number $w\geq1$, we have
\begin{equation}
\label{u41}
\psi \left( \left( \mathbf{k} \overline{1}^w \right)^* \right)=
\psi ( \mathbf{k})\psi \left( (\overline{1}^w)^* \right).
\end{equation}
\end{fact}
\begin{proof} By definition,
  $\xi (\mathbf{k} \overline{1}^w)=n$. Therefore,
$$
(\mathbf{k} \overline{1}^w)^*= \mathbf{k} \overline{1}^{w-1}0
\Longleftrightarrow
w\geq q_0
\Longleftrightarrow
(\overline{1}^w)^*= \overline{1}^{w-1}0.
$$
Hence, $(\mathbf{k} \overline{1}^w)^*=\mathbf{k}(\overline{1}^w)^*$. Observe that by $k_n\ne 1$  we have $n\in D_{\text{Good}}(\mathbf{k}(\overline{1}^w)^*)$. Using this and Fact \ref{u77} we get
$$
\psi ((\mathbf{k} \overline{1}^w)^*)
=
\psi (\mathbf{k}(\overline{1}^w)^*)
=
\psi(\mathbf{k})\psi((\overline{1}^w)^*).
$$
\end{proof}

The second condition in first line of \eqref{u87}
  is to guarantee that $\widehat{\psi}(\mathbf{i})> \psi (\mathbf{i}) $.

 We know by \eqref{a91} that $\widehat{\psi }(\mathbf{i})=\max\{\psi (\mathbf{i}),\psi (\mathbf{i}^-0)\}$.
\begin{fact}\label{w78}
  For every $n\geq 1$ and $\mathbf{i}\in \mathcal{T}_n$
\begin{equation}
\label{w82}
\mathbf{i}^*=
\left\{
\begin{array}{ll}
\mathbf{i}^-0
,&
\hbox{if $\widehat{\psi} (\mathbf{i})>\psi (\mathbf{i})$;}
\\
\mathbf{i}
,&
\hbox{if $\widehat{\psi} (\mathbf{i})=\psi (\mathbf{i})$.}
\end{array}
\right.
\end{equation}
\end{fact}

\begin{proof}
If $i_n\ne 1$,  then by \eqref{a91} we have
$\widehat{\psi}(\mathbf{i})=\psi(\mathbf{i})$.
In this case $n-\xi (\mathbf{i})=0$. Hence,
it follows from \eqref{u87} that
$\mathbf{i}^*=\mathbf{i}$.
So, we may assume that $i_n=1$.  That is,  $n-\xi (\mathbf{i})\geq  1$. Hence,
$\mathbf{i}=\mathbf{i}|_{\xi (\mathbf{i})}\overline{1}^{n-\xi (\mathbf{i})}$. Using Fact \ref{u42}, we get
\begin{equation}
\label{w77}
\psi(\mathbf{i}^*)=\psi(\mathbf{i}|_{\xi (\mathbf{i})})\cdot
\psi \left(\left(  \overline{1}^{n-\xi (\mathbf{i})} \right) ^*\right).
\end{equation}
Using that $\xi (\mathbf{\ell })\in D_{\text{Good}}(\mathbf{i})$
we obtain from Lemma \ref{u77} that
\begin{equation}
\label{w76}
\psi(\mathbf{i})=\psi(\mathbf{i}|_{\xi (\mathbf{i})})\cdot
\psi \left(  \overline{1}^{n-\xi (\mathbf{i})} \right).
\end{equation}
Putting together the last two displayed formulas with
\eqref{u58} and \eqref{u63} we get the assertion of Fact.
\end{proof}
Observe that in virtue of Fact \ref{w78} we have
\begin{equation}
\label{w75}
\widehat{\psi}(\mathbf{i})=\psi(\mathbf{i}^*).
\end{equation}

\subsection{The Proof of Property-\ref{w97} }\label{w79}

\begin{lemma}\label{y23}
  There exists a constant $C_6>1$ such that the
  following holds.
    Let $q\geq 1$ and let
    $\mathbf{j}\in \mathcal{T}_q$ be arbitrary.
  Let  $u,v\in \mathcal{A}$
  such that $\mathbf{j}u,\mathbf{j}v\in \mathcal{T}_{q+1}$.
 Then,
  \begin{equation}
  \label{y22}
  C _{6}^{-1}(|\mathbf{j}|+1)^{-1}
  <
  \frac{\psi (\mathbf{j}v)}{\psi (\mathbf{j}u)}
  <C_6(|\mathbf{j}|+1).
  \end{equation}
  \end{lemma}
\begin{proof}

We may assume that $u\ne v$ and
\begin{equation}
\label{u75}
\text{ if } 0\in\left\{ u,v \right\},
\text{ then }
v=0.
\end{equation}
 Using that any good block ends with a $0$,
this implies that $\chi(\mathbf{j}v)\geq \chi(\mathbf{j}u)$. More precisely,
\begin{equation}
\label{y21}
\left(G(\mathbf{j}v)\ne\emptyset\
\&\
\chi(\mathbf{j}u)\ne \chi(\mathbf{j}v)
\right)
\Longrightarrow
\chi(\mathbf{j}u)<\chi(\mathbf{j}v)=q+1.
\end{equation}
Recall from \eqref{u89} that
\begin{equation*}
\xi:=\xi (\mathbf{j})=\min
\left\{
  k\in[0,q]:\forall \ell >k,\ j_{\ell }=1
 \right\},
\end{equation*}
where $\xi=q$ means that $\mathbf{j}$ does not end with a block of $1$s. If $\xi<q$, then we can write
\begin{equation}
\label{u73}
\mathbf{j}=\mathbf{j}|_{\xi}\overline{1}^{q-\xi}.
\end{equation}
By definition, we get that
\begin{equation}
\label{u72}
\xi\in\left\{ 0 \right\}\cup D_{\text{Good}}(\mathbf{j}).
\end{equation}

By Fact \ref{y76} we get that
\begin{equation}
\label{u71}
\pmb{\omega}\in \mathcal{I}_{\mathbf{j}u},\
\pmb{\tau}\in \mathcal{I}_{\mathbf{j}v}
\Longrightarrow
\Pi (\pmb{\omega}|_{\xi})
=
\Pi (\pmb{\tau}|_{\xi})
=
\Pi (\mathbf{j}|_{\xi})
=
\Pi (\mathbf{j}u|_{\xi})
=
\Pi (\mathbf{j}v|_{\xi}).
\end{equation}
Moreover,
\begin{equation}
\label{u70}
\mathcal{I}_{\mathbf{j}u}=
\left\{
  \pmb{\tau}\overline{1}^{q-\xi}u
:
\pmb{\tau}\in \mathcal{I}_{\mathbf{i}|_{\xi}}
  \right\},
\text{ if } u\in\left\{ 1,3 \right\},
\quad
 \mathcal{I}_{\mathbf{j}0}
 =
 \bigcup _{\ell =0}^{q-\xi}
 \left\{
   \pmb{\tau}\overline{1}^{q-\xi-\ell }0
   \overline{3}^{\ell }:
   \pmb{\tau}\in \mathcal{I}_{\mathbf{i}|_{\xi}}
  \right\}.
\end{equation}
Using that for $u\in \left\{ 1,3 \right\}$, $\chi(\mathbf{j}u)\leq \xi$ it follows from
Fact \ref{y27} that
 \begin{equation}
 \label{u69}
 \psi(\mathbf{j}u)=
 \psi(\mathbf{j}|_{\xi}) p _{1}^{q-\xi}p_u.
 \end{equation}
Using the second part of  \eqref{u70} we get
\begin{equation}
\label{u62}
\psi(\mathbf{j}0)=
 \psi(\mathbf{j}|_{\xi})\cdot
 \sum_{\ell =0}^{q-\xi}
p _{1}^{q-\xi-\ell  }
p _{3}^{\ell }p_0.
\end{equation}
Hence,
\begin{equation}
\label{u68}
\frac{\psi(\mathbf{j}v)}{\psi(\mathbf{j}u)}
=
\left\{
\begin{array}{ll}
\frac{
  \sum_{\ell =0}^{q-\xi}
  p _{1}^{q-\xi-\ell  }
  p _{3}^{\ell }p_0
}{p _{1}^{q-\xi}p_u}
=
\frac{p_0}{p_u}
\sum_{\ell =0}^{q-\xi}
\left( \frac{p_3}{p_1} \right)^{\ell }
,&
\hbox{ if $v=0$;}
\\
\frac{p_v}{p_u}
,&
\hbox{if $v\ne 0$.}
\end{array}
\right.
\end{equation}
That is, by \eqref{y87} we obtain
\begin{equation}
\label{u67}
\frac{p_{\min}}{p_{\max}}
\leq
\frac{\psi(\mathbf{j}v)}{\psi(\mathbf{j}u)}
\leq C_{14}
(|\mathbf{j}|+1-\xi),\quad
u\in\left\{ 1,3 \right\} \text{ and }
v\in\left\{ 0,1,3 \right\}\setminus \{u\},
\end{equation}
where $C_{14}:=
\left\{
\begin{array}{ll}
\frac{p_0p_1}{p_3 (p_1-p_3)}
,&
\hbox{if $p_1>p_3$;}
\\
\frac{p_0}{p_3}
,&
\hbox{if $p_0=p_3$.}
\end{array}
\right.
$
\end{proof}
We obtain from \eqref{u62} that
\begin{equation}
\label{u54}
\xi (\mathbf{i})\leq |\mathbf{i}|-1
\Longrightarrow
\psi (\mathbf{i}^-0)=
\psi (\mathbf{i}|_{\xi (\mathbf{i})})\cdot
\frac{p_0}{p_1}p _{1}^{|\mathbf{i}|-\xi (\mathbf{i})}
\sum_{\ell =0}^{|\mathbf{i}|-\xi (\mathbf{i})-1 }\left(
  \frac{p_3}{p_1}
 \right)^{\ell }
 =
 \psi (\mathbf{i}|_{\xi (\mathbf{i})})\cdot
 a_{|\mathbf{i}|-\xi (\mathbf{i})},
\end{equation}
where the sequence $\left\{ a_q \right\}_{q=1}^{\infty  }$
was defined in \eqref{u43}.
Hence, by Fact \ref{w78}
\begin{equation}
\label{u53}
\mathbf{i}^*\ne \mathbf{i}
\Longrightarrow
\widehat{\psi }(\mathbf{i})=\psi (\mathbf{i}^-0)=
\psi (\mathbf{i}|_{\xi (\mathbf{i})})\cdot
 a_{|\mathbf{i}|-\xi (\mathbf{i})}.
\end{equation}

\begin{fact}\label{u83}
  Let $\mathbf{i}\in \mathcal{T}_n$. If $\xi (\mathbf{i})=n$, then
 $i_n\ne 1$. So,  by definition $\psi(\mathbf{i})=\widehat{\psi}(\mathbf{i})$.
  There exists a
  $C_3>0$ such that
  whenever $\xi (\mathbf{i})\leq n- 1$, then  we have
  \begin{equation}
  \label{u80}
  \psi (\mathbf{i})\leq \widehat{\psi}(\mathbf{i}) \leq
  C_3\cdot (n-\xi(\mathbf{i}))\cdot\psi (\mathbf{i}).
  \end{equation}
   Moreover,
  \begin{equation}
  \label{u79}
  \mathbf{i}\ne \mathbf{i}^*
  \Longleftrightarrow
  \psi (\mathbf{i})< \widehat{\psi}(\mathbf{i})
  \Longleftrightarrow
  \max\left\{
    1,q_0
   \right\}
  \leq n-\xi(\mathbf{i}).
  \end{equation}
  \end{fact}
 Observe that Fact \ref{u83} implies that Property-\ref{w97}
holds.
\begin{proof}
Let $\mathbf{i}\in\mathcal{T}_n$ and
 $\mathbf{j}:=\mathbf{i}^-$.
 It is assumed that $\xi (\mathbf{i})\leq n- 1$. This implies that $\mathbf{i}=\mathbf{j}1$. If $\xi (\mathbf{i})=n-1$, then $\widehat{\psi}(\mathbf{i})=\psi (\mathbf{i})$ by definition. We may assume that
  $\xi (\mathbf{i})\leq n- 2$. Then,
  either $\widehat{\psi }(\mathbf{i})
  =\psi (\mathbf{i})$ or if not
  then by Fact \ref{w78}
  $\mathbf{i}^*=\mathbf{j}0$ and
 by \eqref{u87} and \eqref{u68}
 $$
 1\leq
\frac{\widehat{\psi }(\mathbf{i})}{\psi (\mathbf{i})}
\leq \frac{p_0}{p_1}(n-\xi ).
 $$
 This completes the proof of the first part with the choice of $C_3:=\frac{p_0}{p_1}$.

 To verify the second part we
 first observe that if $\xi (\mathbf{i})\geq n-1$, then $\mathbf{i}=\mathbf{i}^*$ so $\widehat{\psi }(\mathbf{i})=\psi (\mathbf{i})$.
 So, we may assume that $2\leq n-\xi (\mathbf{i})$.
In this case $\widehat{\psi }(\mathbf{i})=\psi (\mathbf{j}0)$ and
$\psi (\mathbf{i})=\psi (\mathbf{j}1)$.
We write $\xi :=\xi (\mathbf{j})$. Then, by
 \eqref{u69} and \eqref{u62} we get
that
\begin{equation}
\label{u56}
\psi (\mathbf{i})<\widehat{\psi }(\mathbf{i})
\Longleftrightarrow
p _{1}^{n-\xi  }
<
\sum_{\ell =0}^{n-1 -\xi }
p _{1}^{n-1-\xi -\ell }p _{3}^{\ell }p_0
\Longleftrightarrow
n-\xi \geq q_0.
\end{equation}
\end{proof}



\subsection{The Proof of Property-\ref{w98} }\label{w79}

\begin{lemma}\label{u65}
  There exists a $C_{10}>1$ such that
  for any $n,m\geq 1$, $\mathbf{k}\in \mathcal{T}_{n+m}$
  \begin{equation}
  \label{u64}
  \widehat{\psi}(\mathbf{k})\leq
  C_{10}\widehat{\psi}(\mathbf{k}|_n)
  \widehat{\psi}(\sigma ^n\mathbf{k}).
  \end{equation}
\end{lemma}
\begin{proof}
  Fix a $\mathbf{k}\in \mathcal{T}_{n+m}$.
  Let
$$
\mathbf{i}=(i_1,\dots  ,i_n):=\mathbf{k}|_n\in \mathcal{T}_n
\quad
\text{ and }\quad
\mathbf{j}=(j_1,\dots  ,j_m):=\sigma ^n\mathbf{k}\in \mathcal{T}_m.
$$

We distinguish two cases:

\noindent\underline{Case 1}
$\xi (\mathbf{k})\leq n-1$.
Then, $\xi (\mathbf{k})=\xi (\mathbf{i})$ and  $1=i_n=j_1=\cdots =j_m$. That is, $\mathbf{j}=\overline{1}^m $.
We distinguish three cases again:
\begin{enumerate}
[{\bf (a)}]
  \item $\mathbf{i}^*\ne \mathbf{i}$ (this implies that $\mathbf{k}^*\ne \mathbf{k}$). Then, $\xi (\mathbf{i})=\xi (\mathbf{k})$. Hence,
\begin{equation}
\label{u52}
\frac{\psi (\mathbf{k}^*)}{\psi (\mathbf{i}^*)}
=
\frac{\psi (\mathbf{i}|_{\xi (\mathbf{i})})\cdot a_{n+m-\xi (\mathbf{i})}}
{\psi (\mathbf{i}|_{\xi (\mathbf{i})})\cdot
a_{n-\xi (\mathbf{i})}}
=
p_{1}^{m}\cdot
\underbrace{
  \frac{\sum\limits_{\ell =0}^{n+m
  -\xi (\mathbf{i})-1 }\left( \frac{p_3}{p_1} \right)^{\ell }}
  {\sum\limits_{\ell =0}^{n
  -\xi (\mathbf{i})-1 }\left( \frac{p_3}{p_1} \right)^{\ell }}}_{r(\mathbf{i},\mathbf{j})}.
\end{equation}
\begin{enumerate}
[{\bf (I)}]
  \item If $m\leq q_0$, then by definition $\mathbf{j}^*=\mathbf{j}$ and then, by Fact \ref{w78}, $\widehat{\psi }(\mathbf{j})=\psi(\mathbf{j})=p _{1}^{m } $  and  $r(\mathbf{i},\mathbf{j})\leq C_5\cdot q_0$ for a $C_5>0$. So, in this case \eqref{u64} holds since in this situation, $\widehat{\psi}(\mathbf{i})=\psi(\mathbf{i}^-0)$ and $\widehat{\psi}(\mathbf{k})=\psi(\mathbf{k}^-0)$.
  \item Assume that $m> q_0$. Then, as we have said, $\mathbf{j}=\overline{1}^m$
  and
  \begin{equation}
  \label{u51}
  \widehat{\psi }(\mathbf{j}) =\psi(\mathbf{j}^-0)=
a_m=
\frac{p_0}{p_1}p _{1}^{m}
\sum_{\ell =0}^{m-1 }
\left( \frac{p_3}{p_1} \right)^{\ell }.
  \end{equation}
  So, in order to verify \eqref{u64} for this case, we have to show that
\begin{equation}
\label{u50}
r(\mathbf{i},\mathbf{j})\leq \text{ Const }\cdot
\sum_{\ell =0}^{m-1 }
\left( \frac{p_3}{p_1} \right)^{\ell }.
\end{equation}
This follows from \eqref{u47}.
\end{enumerate}
  \item $\mathbf{i}^*= \mathbf{i}$ but $\mathbf{k}^*\ne \mathbf{k}$.
Then,
\begin{equation}
\label{u46}
n-\xi (\mathbf{i})<q_0\leq n+m-\xi (\mathbf{i}).
\end{equation}
We write $\xi :=\xi (\mathbf{i})$. Then, $\widehat{\psi}(\mathbf{k})=\psi(\mathbf{k}^-0)$ and $\widehat{\psi}(\mathbf{i})= \psi(\mathbf{i})$,
\begin{equation}
\label{u45}
\frac{\psi (\mathbf{k}^*)}{\psi (\mathbf{i}^*)}
=
\frac{\psi (\mathbf{i}|_{\xi })\frac{p_0}{p_1}p _{1}^{n-\xi+m  }
\sum\limits_{\ell =0}^{n-\xi +m-1 } \left( \frac{p_1}{p_3} \right)^{\ell }}
{\psi (\mathbf{i}|_{\xi })p _{1}^{n-\xi  }}
=
\frac{p_0}{p_1}p _{1}^{m  }
\sum\limits_{\ell =0}^{n-\xi +m-1 } \left( \frac{p_1}{p_3} \right)^{\ell }.
\end{equation}
\begin{enumerate}
[{\bf (I)}]
  \item If $m\leq q_0$, then $\widehat{\psi }(\mathbf{j})=
  p _{1}^{m }$.  Then, $n-\xi +m< 2q_0$. Hence, \eqref{u64} holds in this case.
  \item If $m> q_0$, then $\widehat{\psi }(\mathbf{j})=
  \widehat{\psi }(\overline{1}^m)=a_m=
  \frac{p_0}{p_1}p _{1}^{m}
\sum_{\ell =0}^{m-1 }
\left( \frac{p_3}{p_1} \right)^{\ell }
  $. Using this, \eqref{u45} and the fact that $n-\xi < q_0$ (which holds since $\mathbf{i}^*=\mathbf{i}$) we obtain that \eqref{u64} holds also in this case.
\end{enumerate}
  \item  $\mathbf{i}^*= \mathbf{i}$ and  $\mathbf{k}^*= \mathbf{k}$ (this implies that   $\mathbf{j}^*= \mathbf{j}$ since $\xi (\mathbf{k})\leq n-1$).
  Then,
  $\widehat{\psi }(\mathbf{i})=\psi (\mathbf{i})$,
  $\widehat{\psi }(\mathbf{j})=\psi (\mathbf{j})$ and
  $\widehat{\psi }(\mathbf{k})=\psi (\mathbf{k})$. Hence,
  $$
\frac{\widehat{\psi }(\mathbf{k})}{\widehat{\psi }(\mathbf{i})}
=
\frac{\psi (\mathbf{i}|_{\xi }) p _{1}^{n+m-\xi }}
{\psi (\mathbf{i}|_{\xi }) p _{1}^{n-\xi }}= p _{1}^{m }=\widehat{\psi }(\mathbf{j}).
  $$
This verifies \eqref{u64} for this case.
\end{enumerate}

\noindent\underline{Case 2}
$\xi (\mathbf{k})\geq n$.
In this case $m-\xi (\mathbf{j})=n+m-\xi (\mathbf{k})$. This means that $\mathbf{k}^*\ne \mathbf{k}$ if and only if $\mathbf{j}^*\ne \mathbf{j}$. In particular, $\xi (\mathbf{k})\geq n$ implies that
\begin{equation}
\label{u44}
\mathbf{k}^*=\mathbf{i}\mathbf{j}^*.
\end{equation}
Now we distinguish two cases:
\begin{enumerate}
[{\bf (a)}]
  \item $n\in  D_{\text{Good}}(\mathbf{k}^*) $.
Then, it follows from Lemma \ref{u77} that
$\psi (\mathbf{k}^*)=\psi (\mathbf{i})\psi (\mathbf{j}^*)$. Here we used that by \eqref{u44}, we have $\sigma ^n (\mathbf{k}^*)=
(\sigma ^n\mathbf{k})^*
$. Using that $\psi(\mathbf{i})\leq \widehat{\psi}(\mathbf{i})$ we obtain that \eqref{u64} holds in this case.
\item $n\in  C_{\text{Good}}(\mathbf{k}^*)$.
This means that there is a good block $ [u,v+1]\in \mathcal{G}(\mathbf{k}^*)$ such that
\begin{equation}
\label{u40}
1\leq u\leq n\leq v\leq n+m-1, \quad \text{and} \quad
k_{v+1}=0.
\end{equation}
Using that either $\xi (\mathbf{k}^*)=n+m$ or $k_{\xi (\mathbf{k}^*)}\ne 1 $, we get that
\begin{equation}
\label{u38}
\xi (\mathbf{k}^*)\in D_{\text{Good}}(\mathbf{k}^*).
\end{equation}
Moreover,
\begin{equation}
\label{u37}
u-1<v+1\leq \xi (\mathbf{k}^*),\quad \text{ and } \quad
u-1,v+1 \in D_{\text{Good}}(\mathbf{k}^*)
.
\end{equation}
 Hence, by Lemma \ref{u77} we have
 \begin{equation}
 \label{u36}
 \psi (\mathbf{k}^*)
 =
 \psi(\mathbf{i}|_{u-1})\psi(\overline{1}^{v-u+1}0)
 \psi(\sigma ^{v+1}\mathbf{k}|_{\xi (\mathbf{k})})
 \psi \left(
\left(\overline{1}^{n+m-\xi (\mathbf{k})}\right)^*
  \right),
 \end{equation}
where we remark that the first and last words in \eqref{u36},
$\mathbf{i}|_{u-1}$ and
$(\overline{1}^{n+m-\xi (\mathbf{k})})^*$, respectively, can be the empty words $\flat$. In this case we recall  that $\psi (\flat)=1$.
Using Fact \ref{u49} in the third step we get
\begin{eqnarray}
\label{u35}
\nonumber \psi(\overline{1}^{v-u+1}0)
&=&
\psi(\overline{1}^{n-u+1}\overline{1}^{v-n} 0)
=
a_{(n-u+1)+(v-n+1)}\leq
C_6 a_{n-u+1}a_{v-n+1}
\\
&=&
C_6
\psi (\overline{1}^{n-u}0)
\psi (\overline{1}^{v-n}0)
\leq
C_6
\psi \left(\left( \overline{1}^{n-u+1} \right)^*\right)
\psi (\overline{1}^{v-n}0).
\end{eqnarray}
Now we substitute this into \eqref{u36} and get
\begin{equation}
\label{u33}
\psi (\mathbf{k}^*)
 \leq C_6
 \underbrace{
 \psi(\mathbf{i}|_{u-1})
 \psi \left(\left( \overline{1}^{n-u+1} \right)^*\right)}
 _{\psi (\mathbf{i}^*)}
 \underbrace{
\psi (\overline{1}^{v-n}0)
 \psi(\sigma ^{v+1}\mathbf{k}|_{\xi (\mathbf{k})})
 \psi \left(
\left(\overline{1}^{n+m-\xi (\mathbf{k})}\right)^*
  \right)}
  _{\psi (\mathbf{j}^*)},
\end{equation}
where we used Fact \ref{u42}  and
in the last step we used that $n+m-\xi (\mathbf{k})=m-\xi (\mathbf{j})$. That is, \eqref{u64} holds also in this last possible case.
\end{enumerate}
\end{proof}

We will need the following Claim:

\begin{claim}\label{u29}
  For an arbitrary $\mathbf{i}\in \mathcal{T}_n$ we have
  \begin{equation}
  \label{u28}
  \psi (\mathbf{i}0)\geq p_0\cdot
\widehat{\psi}(\mathbf{i}).
  \end{equation}
\end{claim}
\begin{proof}
  Fix an $\mathbf{i}\in\mathcal{T}_n$.
  \begin{enumerate}
  [{\bf (a)}]
    \item Assume that $i_n\ne 1$.
  Then, $\mathbf{i}^*=\mathbf{i}$ and $n\in D_{\text{Good}}(\mathbf{i}0)$. Then, by Lemma \ref{u77}, we have
$\psi (\mathbf{i}0)=\psi (\mathbf{i})\psi (0)=p_0 \widehat{\psi}(\mathbf{i}) $. So, \eqref{u28} holds in this case.
    \item Assume that $i_n= 1$. Then, $\xi :=\xi (\mathbf{i})\leq n-1$
    and $\xi \in D_{\text{Good}}(\mathbf{i}0)$. Then, by Lemma \ref{u77} we have
    \begin{equation}
    \label{u27}
    \psi (\mathbf{i}0)=
    \psi(\mathbf{i}|_{\xi })
    \psi \left( \sigma ^{\xi }(\mathbf{i}0) \right)
    =
    \psi(\mathbf{i}|_{\xi })
    \psi \left( \overline{1}^{n-\xi }0 \right)
    =
    \psi(\mathbf{i}|_{\xi })\cdot a_{n-\xi +1}.
    \end{equation}
    On the other hand, by Fact \ref{u42}
\begin{equation}
\label{u26}
\widehat{\psi}(\mathbf{i})=
\left\{
\begin{array}{ll}
  \psi(\mathbf{i}|_{\xi }) p _{1}^{n-\xi }
,&
\hbox{if $n-\xi<q_0 $;}
\\
\psi(\mathbf{i}|_{\xi }) a_{n-\xi }
,&
\hbox{if $n-\xi \geq q_0$.}
\end{array}
\right.
\end{equation}
Putting together \eqref{u27} and \eqref{u26} we get that \eqref{u28} holds.
  \end{enumerate}
\end{proof}

\begin{claim}\label{u24}
  Let $\mathbf{j}\in \mathcal{T}_m$ for an $m\geq 1$. Then, we have
  \begin{equation}
  \label{u22}
 \widehat{\psi}(1\mathbf{j})\geq
 p_1 \widehat{\psi}(\mathbf{j}).
  \end{equation}
\end{claim}

\begin{proof}
 First we consider the case when $\mathbf{j}=\overline{1}^m$. This case can be subdivided into the three cases when $m<q_0-1$, $m=q_0-1$ and $m\geq q_0$. Using formula \eqref{u43}, one can easily point out in each of these three cases that \eqref{u22} holds. So, from now on we may assume that there exists a $p\in\left\{ 1,\dots  ,m \right\}$
 such that $j_p\ne 1$. Using this, one can easily see  in the same way as in Case 2 of the proof of Lemma \ref{u65}
 that
\begin{equation}
\label{u23}
\left( 1\mathbf{j} \right)^*=1 \mathbf{j} ^*.
\end{equation}
Observe that
\begin{equation}
\label{u20}
\pmb{\eta}
 \in \mathcal{I}_{\mathbf{j}^*} \Longrightarrow
1\pmb{\eta}
 \in \mathcal{I}_{1\mathbf{j}^*}.
\end{equation}
Namely,
\begin{multline}
\label{}
\pmb{\eta} \in \mathcal{I}_{\mathbf{j}^*} \Longleftrightarrow
\Pi (\pmb{\eta})=\Pi (\mathbf{j}^*)
\Longrightarrow
S_1(\Pi (\pmb{\eta}))=S_1(\Pi (\mathbf{j}^*))
\\
\Longrightarrow
\Pi (1\pmb{\eta})=\Pi (1\mathbf{j}^*)
\Longrightarrow
1\pmb{\eta}
 \in \mathcal{I}_{1\mathbf{j}^*}.
\end{multline}

By definition
\begin{equation}
\label{u21}
\psi(1\mathbf{j}^*)=\sum_{\pmb{\omega}\in \mathcal{I}_{1\mathbf{j}^*}}
p_{\pmb{\omega}}\geq
\sum_{\pmb{\eta}\in \mathcal{I}_{\mathbf{j}^*}}p_{1 \pmb{\eta}}
= p_1 \sum_{\pmb{\eta}\in \mathcal{I}_{\mathbf{j}^*}}p_{1 \pmb{\eta}}
=p_1 \psi(\mathbf{j}^*).
\end{equation}
\end{proof}

\begin{lemma}\label{u32}
  There exist a $C_{11}>0$ such that
  \begin{equation}
  \label{u31}
 \widehat{\psi}(\mathbf{i}01\mathbf{j}) \geq
 C_{11}\widehat{\psi}(\mathbf{i})\widehat{\psi}(\mathbf{j}), \quad
\forall  \mathbf{i},\mathbf{j}\in \Sigma _{A}^{*}.
  \end{equation}
\end{lemma}
\begin{proof}
Given $\mathbf{i}\in\mathcal{T}_n$ and $\mathbf{j}\in\mathcal{T}_m$. Clearly, $\widetilde{\mathbf{k}}:=\mathbf{i}01\mathbf{j}\in\mathcal{T}_{n+m+2}$. Using that $\xi (\widetilde{\mathbf{k}})\leq n+1$. Hence, by the same argument that we used in Case 2, part of the proof of Lemma \ref{u65}, we obtain that
\begin{equation}
\label{u30}
\widetilde{\mathbf{k}}^*=\mathbf{i}0(1\mathbf{j})^*.
\end{equation}
Observe that $n+1\in D_{\text{Good}}(\widetilde{\mathbf{k}}^*)$.
Hence, first using Lemma \ref{u77} and then using
Claims \eqref{u29} and \eqref{u24} we get that
\begin{equation}
\label{u25}
\widehat{\psi}(\widetilde{\mathbf{k}})
=
\psi(\widetilde{\mathbf{k}}^*)=
\psi (\mathbf{i}0)\psi \left( \left( 1\mathbf{j} \right)^* \right)
\geq p_0\widehat{\psi}(\mathbf{i})
p_1 \widehat{\psi}(\mathbf{j}).
\end{equation}
\end{proof}
Putting together Lemmas \ref{u65} and \ref{u32} we obtain that
\begin{corollary}\label{u18}
 For every $t>0$, $\widehat{\psi }^t$ is a quasi-multiplicative potential in the sense of Definition \ref{u19}.
\end{corollary}

This proves that Property-\ref{w98} holds.

\bigskip

\textbf{Acknowledgements} We would like to thank Aljoscha Niemann for his numerous valuable comments, and for calling our attention to the fact that putting together \cite[Corollary 1.12]{KNZ}
and our Theorem \ref{x70} yields an expression for the $L^q$-spectrum over $(0,1)$ for the measures studied in this paper.

\end{document}